\documentclass{article}

\usepackage[condensed]{YanYau}
\usepackage{tikz-cd}
\usepackage{quiver}
\usepackage{booktabs}

\RequirePackage{titling} % Used in fancy and in Title spacing
\RequirePackage[margin=1in]{geometry} %Margins
\RequirePackage[twoside]{fancyhdr}
\RequirePackage{titling} 
\posttitle{\par\end{center}}
\newtheorem{theorema}{Theorem}

\RequirePackage[
	backend=biber,
	style=alphabetic,
	maxnames=6,
	minnames=3,
	maxbibnames=99,
	sorting=nyt,
	backref=true,
	]{biblatex}
	
\DeclareFieldFormat{postnote}{#1}
\DeclareFieldFormat{multipostnote}{#1}

\AtEveryBibitem{
    \clearfield{urlyear}
    \clearfield{urlmonth}
} %This removes the "Visited on" thing 
\title{A Trace--Path Integral Formula over Function Fields}
\author{Yan Yau Cheng}
\date{\vspace{-5ex}}

\usepackage[color=yellow,textcolor=purple]{todonotes}

\begin{document}

\maketitle 

{\abstract{We show that an arithmetic path integral over the \(\ell\)-torsion of a Jacobian \(J[\ell]\) is equal to the trace of the Frobenius action on a representation of the Heisenberg group \(H(J[\ell])\), up to an explicitly determined sign. This is an arithmetic analogue of trace--path integral formulae which arise in quantum field theory, where path integrals over a space of sections of a fibration over a circle can be expressed as the trace of the monodromy action on a Hilbert space.}}

\begingroup
\begin{NoHyper}
\let\thefootnote\relax\footnotetext{University of Edinburgh, \href{mailto:yanyau.cheng@ed.ac.uk}{yanyau.cheng@ed.ac.uk} }
\end{NoHyper}
\endgroup

\section{Introduction}
In a topological quantum field theory, one often cares about path integrals of an action functional \(A\) over a space \(\cF\) of sections of a fibration \(M\times [0,1]/F \to S^1\) (determined by a monodromy \(F\circlearrowright M\)). The path integral of \(A\) over this space can in fact be computed by instead looking at the action of \(F\) on a Hilbert space \(\cH\) of functions as 
\[\Tr(F|\cH)=\int_{\cF} e^{iA(\gamma)}d\gamma.\]

The purpose of this paper is to explain how each of these concepts has an arithmetic analogue, and to prove a version of this trace path integral formula for \(J\) the Jacobian of a smooth projective curve \(X\) over the finite field \(\FF_q\). 
\begin{center}
\begin{tabular}{rcl}
Fibration \(M\times [0,1]/F \to S^1\) &::& Jacobian \(\ell\)-torsion \(J[\ell]\to \Spec \FF_q\)\\
Space of sections \(\cF\) &::& Rational points \(J[\ell](\FF_q)\)\\
Action functional \(A\) &::& Pairing \(A\) arising from class field theory \\
Geometric quantisation of phase space \(\cH\) &::& Global sections of theta line bundle \(\cH\) \\ 
Monodromy action \(F\) &::& Frobenius action \(\Fr_q\)\\
\end{tabular}
\end{center}

Furthermore, the arithmetic path integral of \(A\) can be expressed as the trace of \(\Fr_q\) on \(\cH\) just as in the physical setting.
\begin{theorema}[\Cref{thm:MainTheoremFull}]
	Let \(J\) be the Jacobian of a genus \(g\) curve \(X\) over a finite field \(\FF_q\). 
	Let \(\ell\) be an odd prime satisfying \(q\equiv 1\pmod {\ell^2}\). If \(\Fr_q\) acts semisimply on the \(\FF_\ell\)-vector space \(J[\ell]\), and \(J(\FF_q)\) has no points of order \(\ell^2\), then we have the following equality
	\[\tr(\Fr_q|\cH) =\legendre{(-1)^g \det'(1-\Fr_q) \det(A)}{\FF_\ell}\sum_{\gamma\in J[\ell](\FF_q)}e^{2\pi iA(\gamma)},\]

	where \(\legen{\cdot}{\FF_\ell}\) denotes the Legendre symbol, and \(\det'(1-\Fr_q)\) is the regularised determinant of the linear map \(1-\Fr_q\) on vector space \(J[\ell]\). That is, this is the determinant of \(1-\Fr_q\) after quotienting the space \(J[\ell]\) by its kernel. 
\end{theorema}

Regularised determinants are used in mathematical physics to assign finite values on self-adjoint operators of infinite dimensional spaces. Even in finite dimensional vector spaces, the removal of the kernel can be viewed as a kind of regularisation. The occurence of regularised determinants in our formula strengthens the analogy between arithmetic path integrals and physical path integrals. 

Determining signs in arithmetic explicitly can often be an intricate and delicate task, and the main contribution of this paper is the explicit Legendre symbol above. Additionally we also show that this statement is true as long as the Frobenius acts semisimply on \(J[\ell]\). A weaker version of \cref{thm:MainTheoremFull} was initially an unpublished result of Minhyong Kim and Akshay Venkatesh, where they show the formula up to an undetermined sign in the special case where the vector space \(J[\ell]\) has an invariant Lagrangian with respect to the Frobenius action. 

This result adds to the series of analogies between topology and arithmetic first noticed by Mazur in \cite{MazurRemarksAlexanderPolynomial} and expanded upon in detail in \cite{KnotsAndPrimes}. For \(X\) a smooth projective curve over a finite field \(\FF_q\), it is natural to compare \(\bar{X}:= X\times_{\FF_q} \Spec \bar{\FF}_q\) with a smooth compact Riemann surface \(\Sigma\). On the other hand, \(X\) itself has more in common with a three-manifold \(N\) -- for instance \(X\) has \'etale cohomological dimension 3 and both sit in Cartesian squares
\[\begin{tikzcd}
	{\bar{X}} & X \\
	{\Spec \bar{\FF}_q} & {\Spec \FF_q}
	\arrow[hook, from=1-1, to=1-2]
	\arrow[from=1-1, to=2-1]
	\arrow[from=1-2, to=2-2]
	\arrow[from=2-1, to=2-2]
\end{tikzcd}\qquad\qquad \begin{tikzcd}[cramped]
	\Sigma & N \\
	1 & {S^1}
	\arrow[hook, from=1-1, to=1-2]
	\arrow[from=1-1, to=2-1]
	\arrow[from=1-2, to=2-2]
	\arrow[from=2-1, to=2-2]
\end{tikzcd}\]
where the map \(N\to S^1\) is a fibration and \(\Spec \bar{\FF}_q\) is a point. 
Under the knots and primes analogy, which views \(S^1\) as being analogous to \(\Spec \FF_q\), we can view \(X\) as analogous to a three-manifold fibred over a circle, with fibres \(\bar{X}\).

Moreover, we can write \(N\) as a mapping torus \[N = \Sigma\times[0,1]/F\] where \(\Sigma\times \{0\}\) is identified with \(\Sigma\times \{1\}\) via the monodromy action \(F:\Sigma \to \Sigma\). In a similar manner we can view \(X\) as a mapping torus, where the automorphism \(\Fr_q\in \pi_1(\Spec \FF_q)\) act on \(\bar{X}\) via the Frobenius automorphism
\[X \text{ `}=\text{' } \bar{X}\times[0,1]/\Fr_q.\]

Quantum field theories are often defined as an integral over the `space of all paths' \cite{HallQuantumBook}. For example, let the manifold \(M\) be the phase space, let \(P(x,y)\) denote the space of all paths from \(y\) to \(x\) in \(M\), 
\[P(x,y):= \{\gamma:[0,T] \rightarrow M | \gamma(0)=y, \gamma(T)=x\}.\]
An action functional is a function \(A: P(x,y) \rightarrow \CC\), then the kernel function can be expressed as an integral over all paths in \(P(x,y)\)
\[K_T(x,y)= \int_{P(x,y)}e^{iA(\gamma)}d\gamma.\]
If we informally consider the kernel function \(K_T(x,y)\) as a `matrix' with infinite dimensions, then we can write
\[\Tr(K_T)= \int_x\int_{P(x,x)}e^{iA(\gamma)}d\gamma dx= \int_\Omega e^{iA(\gamma)}d\gamma\] 
where \(\Omega\) is defined to be the space of all loops in \(M\). 

In a similar manner which will be explained in detail in the next section, when a field theory is \textit{topological} in nature (i.e. the theory does not depend on the metric of the manifold), then it is possible to express the trace of a monodromy action \(F\colon M\rightarrow M\) as a path integral over the space of sections \(\cF\) of the bundle \(([0,T]\times M)/F\rightarrow S^1\)
\[\Tr(F)=\int_{\cF} e^{iA(\gamma)}d\gamma.\]

Arithmetic path integrals over number fields have been introduced and computed in \cite{CKKPPYAbelianArithmeticChern}, \cite{CKNoteAbelianArithmetic}, and \cite{CCKKPYPathIntegralsPadic}, and this paper introduces a function field analogue of arithmetic path integrals. A related construction also motivated by the same analogies can be found in \cite{AVSymplecticLFunctions}.

Viewing the curve \(X\) as a `3-dimensional spacetime', the Jacobian \(\bar{J}=\Jac(\bar{X})\) can be thought of as a space of fields on \(X\). For large values of \(\ell\), the Jacobian torsion \(J[\ell]\) can be viewed as an approximation of the space \(\bar{J}\). Thus in this paper we take the arithmetic analogue of the phase space of \(X\) to be \(\ell\)-torsion \(J[\ell]\). (Unfortunately an assumption we make in this paper is that the group of \(\ell\)th roots of unity are contained in the base field \(\FF_q\), so \(\ell\) cannot be arbitrarily large, one interesting direction of further research would be to generalise the results of this paper with the \(\mu_\ell\subset \FF_q\) condition removed.)

Just as how \(X\) can be viewed as a fibred three-manifold, we can in a similar manner view the \(J[\ell]\) as being fibred over the circle
\[\begin{tikzcd}
	{\bar{J}[\ell]} & J[\ell] \\
	{\Spec(\bar{\FF}_q)} & {\Spec(\FF_q).}
	\arrow[hook, from=1-1, to=1-2]
	\arrow[from=1-1, to=2-1]
	\arrow[from=1-2, to=2-2]
	\arrow[from=2-1, to=2-2]
\end{tikzcd}\]
Moreover, we can view a rational point \(\gamma\in J[\ell](\FF_q)\) as a section over this fibre bundle over \(S^1\)
\[\begin{tikzcd}
	{J[\ell]} \\
	{\Spec \FF_q.}
	\arrow[shift right=2, from=1-1, to=2-1]
	\arrow["\gamma"', shift right=2, from=2-1, to=1-1]
\end{tikzcd}\]
This is why in our setting, the integral over the space of all sections is instead replaced with a discrete sum over all \(\FF_q\)-rational points of \(J[\ell]\)
\[\sum_{\gamma\in J[\ell](\FF_q)}e^{2\pi iA(\gamma)}.\]
On the other hand, the arithmetic analogue of the Hilbert space \(\cH\) is given by the space of global sections of a tensor power of the theta line bundle \(\Theta\), in analogy to the geometric quantisation construction arising from physics. This will be explained in detail in \cref{sec:Physics_Quantisation,sec:Arithmetic_Trace}.

\subsubsection*{Outline of Paper}

In \cref{sec:PhysicsBackground} we will outline the physical background behind the trace--path integral formula that motivates our main theorem, in particular in \cref{sec:Physics_Quantisation} we will discuss a method to obtain a Hilbert space \(\cH\) from the phase space \(M\) via a process called \emph{geometric quantisation}. 

The main theorem itself is proved by evaluating each side of the equality separately and then comparing the two sides. 
In \cref{sec:Arithmetic_Trace} we evaluate the trace side of the equality. We first define the space \(\cH\) which is a representation of the Heisenberg group \(H(J[\ell])\). In \cref{sec:Heisenberg_Rep} and \cref{sec:Symplectic_Action} we express the \(\Fr_q\) action on \(\cH\) explicitly using machinery from \cite{GH_QuantisationFiniteField}. In \cref{sec:Symplectic_Linear_Algebra} we study the \(g\)-invariant spaces of symplectic vector space for a symplectomorphism \(g\). We finally compute the trace of the Frobenius action on \(\cH\) in \cref{sec:Trace_Sum_Calculation} by decomposing \(\bar{J}[\ell]\) into a direct sum of \(\Fr_q\) invariant symplectic subspaces and computing the traces in those spaces separately. 

In \cref{sec:Arithmetic_PathIntegrals} we will compute the path integral side of the equality. We first properly define the arithmetic action \(A\) in terms of a pairing in geometric class field theory. Then in \cref{sec:Abelian_CS} we show that the arithmetic action \(A\) actually coincides with a function field analogue of the abelian arithmetic Chern-Simons action defined in \cite{AbelianACS2019}. Finally we evaluate the path integral in \cref{thm:Path_Integral}.

We prove our main theorem in \cref{sec:Conclusion} by combining the results from the two previous sections. In \cref{sec:Examples} we will provide computational examples of the trace--path integral formula, and outline some future research directions.

\subsubsection*{Acknowledgements}

This paper is written as part of my PhD research at the University of Edinburgh. I thank my advisor Minhyong Kim for his invaluable guidance throughout this project, and for introducing me to arithmetic field theories. Also many thanks to Akshay Venkatesh and Minhyong Kim for their permission to base my result as a strengthening of their unpublished work. Thanks also to Masanori Morishita, Dohyeong Kim, and Will Sawin for helpful comments on an earlier draft of the paper. I'd also like to thank Eric Ahlqvist, Danil Ko\v{z}evnikov, Subrabalan Murugesan, Yuji Okitani, Parth Shimpi, Jiacheng Tang, and Harvey Yau, for many helpful discussions. Finally, many thanks the anonymous referee for their detailed report and comments, and in particular for pointing out that the Hamiltonian phase space path integral formalism is the better analogy to use in \cref{sec:PhysicsBackground} rather than the Lagrangian path integral. 

\section{Physical Background} \label{sec:PhysicsBackground}

In this section we go over the physics background that motivates our main result. References for this section include \cite{HallQuantumBook}, \cite{ZeidlerQuantumFieldTheoryI}, \cite{FSGaugeFieldsIntroduction}, \cite{Zinn-JustinPathIntegralsQuantum}, \cite{ASCondensedMatterField}, \cite{EMRVMathematicalFoundationsGeometric}, and \cite{CLLQuantizationKahlerManifolds}.

\subsection{The Path Integral Formalism}

Let the manifold \(M\) denote the \textit{phase space} in classical mechanics, which encodes the state of a physical system at a given time. This is typically an even-dimensional symplectic manifold consisting of data about the positions \(q\) and momenta \(p\) of particles. A \textit{configuration space} is a Lagrangian submanifold $L$ of $M$, roughly corresponding to the position coordinates. (There are also other models for the space of wave functions, but this will be the model that we use in this paper.) 

In a quantum system, the phase space is replaced with a Hilbert space \(\cH\), which encodes all the possible states of a quantum system. The process in which the Hilbert space \(\cH\) is obtained from the phase space \(M\) is called \textit{quantisation}, and we will discuss this process in more detail in \cref{sec:Physics_Quantisation}.

One typical example is when the phase space is a cotangent bundle $M=T^*X$ for some manifold $X$, in which case we could take the configuration space to be $L=X$ viewed as the zero-section in $M$, and \(\cH=L^2(X)\) is the \textit{quantisation} of the symplectic manifold \(T^*X\).

Setting Planck's constant to be 1, the Schr\"odinger's equation can be written as
\[\frac{d\psi}{dt}=-iH\psi,\]
where the \(\psi\) is a time-dependent wave function $\psi\in \cH$, and $H$ is a self-adjoint operator on \(\cH\) called the Hamiltonian, representing energy. 
The path integral formalism arises when representing time evolution according to Schr\"odinger's equation as
\[[e^{-iHT}\psi] (x)=\int_{y\in L} K_T(x,y)\psi(y)dy\]
for some kernel function $K_T(x,y)$. This is an integral over the configuration space \(L\).
 
The classical Feynmannian path integral formalism interprets the kernel function \(K_T(x,y)\) as an integral over a space of paths in the configuration space \(L\). However, the Hamiltonian phase space path integral formalism detailed in \cite[Section 2.1]{FSGaugeFieldsIntroduction} and \cite[Chap. 10]{Zinn-JustinPathIntegralsQuantum} instead shows that the same kernel function \(K_T(x,y)\) can instead be written as an integral over paths in the phase space \(M\). More precisely, 
\[K_T(x,y)=\int_{P(x,y)} e^{iA(\gamma)}d\gamma.\]
at
Here, $A(\gamma)$ is the classical action defined on paths $\gamma\colon [0,T]\to M$, and \(P(x,y)\) is a space of paths in the phase space \(M\), where the position co-ordinates of the path start and end at \(y\) and \(x\) respectively. More precisely, 
\[P(x,y)=\{\gamma\colon [0,T]\to M:q(\gamma(0))=y, q(\gamma(T))=x\}.\]

We can interpret the kernel function \(K_T(x,y)\) as a `matrix' with infinite dimensions, and under this viewpoint one has informally
\[\Tr( e^{-iHT})=\int K_T(x,x)dx.\]

But since we have interpreted \(K_T(x,y)\) as an integral over paths in the phase space \(M\), we can substitute the kernel function for the path integral to obtain
\[\Tr( e^{-iHT})=\int_{x\in L}\int_{P(x,x)} e^{iA(\gamma)}d\gamma dx.\]

Here, the integral is taken over all paths \(\gamma\) in our phase space \(M\) where the position co-ordiantes of the end points are equal. In particular the momenta co-ordinates of the paths need not be equal on the end points. 

% The paths that satisfy the equations of motion \(p=m\frac{dq}{dt}\) (also called the \emph{classical paths}) would also be periodic in the momenta co-ordinate. The \emph{semi-classical approximation} is the integral over just the classical paths instead of the all paths \(P(x,y)\). In certain theories where the Hamiltonian action is quadratic in position and momenta (such as the quantum harmonic oscillator), the semi-classical approximation of \(K_T\) is in fact equal to the full path integral of \(K_T\). 

% So suppose we are in the setting where the semi-classical approximation is equal to the full path integral, in this case we can write
% \[K_T(x,y)=\int_{P_c(x,y)} e^{iA(\gamma)}d\gamma,\]
% where we denote by \(P_c(x,y)\) the set of paths where the momenta co-ordinates should also agree. 
% \[P_c(x,y)=\{\gamma\colon [0,T]\to M:q(\gamma(0))=y;\quad q(\gamma(T))=x, p(\gamma(0))=p(\gamma(T))\}.\]

% Then we would have that the trace is equal to 
% \[\Tr(e^{-iHT})=\int_{x\in L}\int_{P_c(x,x)} e^{iA(\gamma)}d\gamma dx. = \int_\Omega e^{iA(\gamma)}d\gamma,\]
% where \(\Omega\) denotes the space of all loops \(S^1\rightarrow M\). 

\subsection{Twists and Trace of Monodromy}

Note that any map from $S^1$ to $M$ can be viewed as a section of the trivial bundle
\[S^1\times M\to S^1.\]
From a geometric point of view, it is natural to `twist' this situation slightly and  integrate over a space of sections of a non-trivial bundle
\[Y=([0,T]\times M)/ F\to S^1,\]
where $Y$ is a mapping torus, and the monodromy map $F\colon M\to M$ is used to glue $T\times M$ to $0\times M$.
Sections of this fibre bundle can be identified with $c\colon[0,T]\to M$ such that $Fc(T)=c(0).$

The diffeomorphism $F$ acts on functions in \(\cH\) via $$F\psi(x)=\psi (F^{-1}x).$$
Then we can write
\[[Fe^{-iHT}\psi] (x)=[e^{-iHT}\psi] (F^{-1}x)=\int K_{T}(F^{-1} x,y)\psi(y) dy.\]
That is, $K_{T}(F^{-1} x,y)$ is the integral kernel for the operator $Fe^{-iHT}$.
Recalling from earlier that $K_{T}(x,y)=\int_{P(x,y)} e^{iA(q)}dq$, we have: 
\[\Tr(Fe^{-iHT})=\int K_{T}(F^{-1} x,x)dx=\iint_{P( F^{-1}x,x)} e^{iA(\gamma)}d\gamma dx.\]

Similarly to the un-twisted case, we can write the kernel function as an integral over just the classical paths, which again can be expressed as an integral over the space of sections of the fibre bundle. In particular, when the theory is {\em topological} so that the Hamiltonian is zero, we get
$$\Tr(F)=\int_{\cF} e^{iA(\gamma)}d\gamma$$
where we use $\cF$ to denote the space of paths in the twisted bundle $Y\to S^1$ such that the position co-ordinates form a loop.

\subsection{Geometric Quantisation} \label{sec:Physics_Quantisation}

As mentioned above, quantisation refers to the process in which a phase space is replaced with a quantum Hilbert space. More precisely, quantisation is a process
\[(M, \omega) \mapsto \cH.\]
Which takes a symplectic manifold \((M, \omega)\) to a Hilbert space \(\cH\). This is also accompanied with a process that sends functions \(f\) on \(M\) to operators \(\hat{f}\) on the space \(\cH\). There are various ways in which quantisation can be performed, and usually one needs more data than just the symplectic manifold \(M\). In this section we will briefly sketch the process of \textit{Geometric Quantisation}. 

Suppose that \((M, \omega)\) is a symplectic manifold. Assume that the symplectic form is in \(\omega \in H^2(M, \ZZ)\) and lies in the image of the Chern map \(c_1\colon H^1(M, \mathcal O_M^\times) \rightarrow H^2(M, \ZZ)\). Then we can construct a line bundle \(\cL\) such that \(\omega\) is its Chern class
\[c_1(\cL)= \omega.\]
Then the \textit{pre-quantisations} of \(M\) can be viewed as the global sections of this line bundle
\[\cH^\rm{pre}_k = H^0(M, \cL^{\otimes k})= \Gamma(M, \cL^{\otimes k}).\]
The prequantisation is `too big' as a space, in order to obtain the quantisation one often takes a polarisation of \(M\) and define the quantisation to be the polarised sections of the line bundle instead. 

If we additionally suppose \(M\) is K\"ahler with complex structure \(J\), then the K\"ahler-polarised sections is simply the space of holomorphic sections of \(\cL^{\otimes k}\)
\[\cH_k = \Gamma_\rm{hol}(M, \cL^{\otimes k}).\]
This means that the sections of the line bundle \(\cL^{\otimes k}\) are precisely the algebraic sections of the line bundle. 

In particular if \(M=X_{hol}\) is a complex projective variety and \(\cL\) is a line bundle on \(X\) with Chern class \(\omega\), then we can take the quantisation of \(M\) to be the space of sections \(\cH_k=\Gamma(X,\cL^{\otimes k})\). 

As an example, when \(X=\RR^2=\CC\), we can take
\[\cH=L^2(\RR,dx) \quad \text{ or } \quad \cH=L^2_\rm{hol}(\CC, e^{-|z|^2 idzd\bar{z}}).\]
In fact, by the \textit{Stone-von Neumann Theorem}, since the actions of the center are equal on these two spaces, they are isomorphic as representations of the Heisenberg algebra. 

\section{Trace of Frobenius in the Arithmetic Hilbert Space} \label{sec:Arithmetic_Trace}

Let \(W=W(\bar{\FF}_q)\) be the Witt vectors over \(\bar{\FF}_q\) (This is a mixed characteristic ring that is an infinite unramified extension of \(\ZZ_p\) with \(\Gal(W/\ZZ_q)\cong \Gal(\bar{\FF}_q/\FF_q)\cong \hat\ZZ\)). Throughout this section we fix an embedding \(W\hookrightarrow \CC\). 

Once again let \(X/\FF_q\) be a curve over a finite field. We fix an algebraic closure \(\bar{\FF}_q\) denote by \(\bar{X}/\bar{\FF}_q\) to be the base change of \(X\) to the algebraic closure. The goal of this section is to define the arithmetic analogue of the quantisation \(\cH\) of our phase space \(\bar{J}[\ell]\), and then to compute the trace of the Frobenius (which is the arithmetic analogue of the monodromy) on this space. 

Define \(Y\) to be a lift of \(\bar{X}\) to \(W\). That is, \(Y\) is a curve over \(W\) such that the reduction of \(Y\) mod \(p\) is \(\bar{X}\). Since we have fixed an embedding of \(W\hookrightarrow \CC\), we can view \(Y\) as a complex curve upon base changing to \(\CC\). Letting \(J_Y\) be the Jacobian of \(Y\times_W \CC\), the Jacobian comes equipped with a canonical principal polarisation coming from the intersection pairing on \(Y\times_W \CC\), which induces a Riemann form on \(J_Y\). Let \(\Theta\rightarrow J_Y\) be the theta line bundle associated to this polarisation. We fix an odd prime \(\ell\) such that \(q\equiv 1 \mod \ell\) (i.e. so that \(\mu_\ell \subset \FF_q\)), then we define:

\begin{definition}\label{def:cH_Quantisation}
	We define the arithmetic analogue of the quantisation of \(J[\ell]\) to be \(\cH\), the global sections of the theta line bundle
	\[\cH \defeq \Gamma(J_Y,\Theta^{\otimes \ell}).\]
\end{definition}

Let \(\anglebrackets{\cdot,\cdot}\colon J[\ell]\times J[\ell]\rightarrow \mu_\ell(\bar{\FF}_q)\) be the Weil pairing on \(J[\ell]\) induced by the canonical principal polarisation, this is a symplectic form on the \(2g\)-dimensional \(\FF_\ell\) vector space \(J[\ell]\). 

By \cite[Lemma 2.6]{OSZUniformityConjecturesAbelian}, it is shown that the Weil pairing when viewed as a class in \(H^2(J_Y, \mu_\ell)\) coincides with the Chern class of \(c_1(\Theta)\) up to sign. This is why we view the space of sections \(\Gamma(J_Y,\Theta^{\otimes \ell})\) as an analogue of the quantisation of our phase space \(J[\ell]\). 

Since \(X\) is defined over \(\FF_q\) and \(\mu_\ell \subset \FF_q\), the Frobenius \(\Fr_q\) acts trivially on \(\mu_\ell(\bar{\FF}_q)\) and
\[\anglebrackets{\Fr_q(x),\Fr_q(y)}=\Fr_q\anglebrackets{x,y}=\anglebrackets{x,y}.\]

Thus \(\Fr_q\) is a symplectomorphism in the symplectic group \(\Sp(J[\ell])\). We will show in \cref{thm:UniqueRep} that \(\cH\) is a representation of the symplectic group \(\Sp(J[\ell])\). Thus we can consider the action of \(\Fr_q\) on \(\cH\) and our goal is the compute the trace of this action
\[\Tr(\Fr_q | \cH).\]

Under some mild assumptions, we have the following explicit formula for the trace of the Frobenius action on \(\cH\).

\begin{theorem}\label{thm:TraceFormula}
	Suppose that the Frobenius \(\Fr_q\) acts semi-simply on the space \(J[\ell]\). Then
	\[\tr(\Fr_q | \cH)=\legen{(-1)^{g-\dim_{\ell} J[\ell](\FF_q)/2}\det'(1-\Fr_q)}{\FF_\ell}\sqrt{|J[\ell](\FF_q)|}.\]
	Where \(\legen{\cdot}{\FF_\ell}\) is the Legendre symbol, and \(\det'\) is the regularised determinant. 
\end{theorem}

This is a strengthening of an unpublished result by Minhyong Kim and Akshay Venkatesh where the sign in the Legendre symbol was undetermined and there is an added assumption that there must exist a Frobenius--invariant Lagrangian subspace of \(J[\ell]\).

An outline of the proof is as follows: We will show in \cref{thm:UniqueRep} that the representation \(\cH\) is in fact the unique representation  of the Heisenberg group \(H(J[\ell])\) with identity central character. We then use the machinery from \cite{GH_QuantisationFiniteField} outlined in \cref{con:Rep_From_Lagrangian} and \cref{con:Intertwining_Morphism} to write down this representation \(\fH(V)\) and the action of a symplectomorphism \(g\in \Sp(V)\) on this representation explicitly. We derive an expression for the trace of such an action in \cref{thm:Trace_Any_Lagrangian} when given a fixed Lagrangian \(M\) of \(V\). Finally, in \cref{lem:Trace_sum_of_spaces,thm:trace_separable_char_poly,thm:Trace_Char_Poly} we decompose the symplectic vector space \(V\) into a direct sum of \(g\)-invariant symplectic subspaces and compute the trace of \(g\) action on \(\fH(V)\) as a product of the traces of the \(g\) action on the smaller spaces to complete the proof.

\subsection{Representations of the Heisenberg Group}\label{sec:Heisenberg_Rep}

Before we prove \cref{thm:TraceFormula} we first establish some theory about the representations of Heisenberg groups over finite fields. Let \(V/\FF_\ell\) be a symplectic vector space with symplectic form \(\omega:V\times V \rightarrow \FF_\ell\). 

\begin{definition}
	Define \(H(V)\) to be the \textit{Heisenberg group}, a central extension of \(V\) by \(\FF_\ell\). Explicitly as a set: \(H(V)=\FF_\ell\times V\) and the group operation is given by
	\[(\lambda,a)\circ(\mu,b)= \left(\lambda + \mu + \frac12\omega(a,b), a+b\right).\]
	This group has center \[Z:=Z(H(V))=\{(z,0):z\in \FF_\ell\}.\]
	There is also an action by the finite symplectic group \(\Sp(V)\) on \(H(V)\) via its action on the \(V\) part.
\end{definition}

A finite analogue of the Stone-von Neumann property, proven in \cite[1.1]{GHQuantizationSymplecticVector}, states that for any non-trivial character \(\psi:Z\rightarrow \CC^\times\), there exists a unique (up to isomorphism) irreducible complex unitary representation \(\fH=\fH(V,\psi)\) of \(H(V)\), \(\pi:H(V) \rightarrow \mathfrak H\) such that
\[\pi|_Z(z)= \psi(z)\cdot \Id_{\mathfrak H}.\]

We write \(\fH=\fH(V)=\fH(V,\psi)\) if it is clear from context what \(\psi\) or \(V\) are. 
Now we construct the representation \(\fH(V,\psi)\) explicitly following the construction in section 2.1 of loc. cit. 

\begin{construction} \label{con:Rep_From_Lagrangian}
	Given a character \(\psi: \FF_\ell \rightarrow \CC^\times\), let \(\CC(H(V),\psi)\) denote the space of (set-valued) functions \(f:H(V) \rightarrow \CC\) such that for \(h\in H(V)\) and \(z \in \FF_\ell\) (where \(z\) is viewed as an element of the center of \(H(V)\)) 
	\[f(z\cdot h)= \psi(z)f(h).\]
	Let an oriented Lagrangian \(M^\circ\) be the langrangian \(M\) along with a choice of non-zero vector \(o_M\in \bigwedge^{top} M\). Then consider the vector subspace \(C_{M^\circ}\) of \(\CC(H(V),\psi)\) consisting of functions \(f\) such that any \(m\in M\) acts trivially
	\[f(m\cdot h)= f(h).\]
	This vector space has a right \(H(V)\) action given by right translation
	\[\pi_M(h)[f](h')=f(h'\cdot h).\]
\end{construction}

\begin{lemma} \label{thm:UniqueRep}
	Upon fixing an isomorphism \(\FF_\ell \to \mu_\ell\subset \CC^\times\) which sends \(1\) to \(\zeta\), the Weil pairing is a symplectic form on \(J[\ell]\), and the space \(\cH\) is the unique (up to almost unique isomorphism given by scalar multiplication by \(a\in \FF_\ell\)) irreducible representation of \(H(J[\ell])\) with the identity central character \(\psi(z)=\zeta^z\). 
\end{lemma}

\begin{proof}
	By an analogue of the Stone Von-Neumann Theorem \cite[Cor 6.4.3, Ex 6.10.3]{BLComplexAbelianVarieties}, the global sections of any line bundle \(\cL\) is a unique irreducible representation of the Theta group \(\mathcal G(\cL)\) with identity central character. We will now show that the Heisenberg group \(H(J[\ell])\) is a subgroup of the theta group \(\mathcal G(\Theta^{\otimes \ell})\), and the restricted representation is still irreducible.

	By \cite[p. 225]{MumfordAbelianVarieties} the theta group is a central extension of the group \(K(\cL)= \Ker(\phi_\cL)\) where \(\phi_\cL\colon J_Y \rightarrow \Pic^0 J_Y\) is given by \(x\mapsto T_x^*\cL \otimes \cL\)
	\[0\rightarrow \CC^\times \rightarrow \cG(\cL) \rightarrow K(\cL) \rightarrow 0.\]
	Since \(\Theta\) is a principal polarisation, it follows that \(\phi_\Theta\) is an isomorphism and \(K(\Theta)= 0\). Furthermore, since \(\phi_{\Theta^{\otimes \ell}}= \ell \phi_\Theta\) (\cite[p. 60]{MumfordAbelianVarieties}), it follows that \(K(\Theta^{\otimes \ell})= \Ker (\phi_{\Theta^{\otimes \ell}}) = J_Y[\ell]\).

	On the other hand, the Heisenberg group \(H(J_Y[\ell])\) is an extension of \(J_Y[\ell]\) by \(\mu_\ell\). Additionally it is known from \cite[p. 228 (5)]{MumfordAbelianVarieties} that the commutator pairing on \(K(\Theta^{\otimes \ell})\) co-incides with the Weil pairing on \(J_Y[\ell]\) induced by polarisation \(\Theta^{\otimes \ell}\). 
	
	Thus there are injections
	\[\begin{tikzcd}
		0 & {\mu_\ell} & {H(J_Y[\ell])} & {J_Y[\ell]} & 0 \\
		0 & \CC^\times & {\cG(\Theta^{\otimes\ell})} & {K(\Theta^{\otimes\ell})} & 0.
		\arrow[from=1-1, to=1-2]
		\arrow[from=1-2, to=1-3]
		\arrow[hook', from=1-2, to=2-2]
		\arrow[from=1-3, to=1-4]
		\arrow[hook', from=1-3, to=2-3]
		\arrow[from=1-4, to=1-5]
		\arrow[from=2-1, to=2-2]
		\arrow[from=2-2, to=2-3]
		\arrow[from=2-3, to=2-4]
		\arrow[equals, from=2-4, to=1-4]
		\arrow[from=2-4, to=2-5]
	\end{tikzcd}\]

	So \(\cH\) is a representation of \(H(J_Y[\ell])\) with the identity central character. To see that the representation is irreducible, suppose that \(\cH \cong U\oplus V\) is reducible as a \(H(J_Y[\ell])\)-representation. But any element of \(\cG(\Theta^{\otimes\ell})\) will differ from an element of \(H(J_Y[\ell])\) by an element in \(\CC\), but since elements in \(\CC\) acts on \(\cH\) by scaling, this means that \(U\) and \(V\) are also invariant subspaces of \(\cG(\Theta^{\otimes\ell})\), contradicting that \(\cH\) is an irreducible representation of \(\cG(\Theta^{\otimes\ell})\).

	The uniqueness of this representation follows from another analogue of the Stone Von-Neumann Theorem, see \cite[Thm 1.1]{GHQuantizationSymplecticVector}.

	Finally, since we have picked \(q\equiv 1 \mod \ell\), the reduction map from \(W\) to \(\FF_q\) induces an isomorphism of the \(\ell\)-torsion subgroups \(J[\ell]\) and \(J_Y[\ell]\) as Galois modules, which then induces an isomorphism between \(H(J_Y[\ell])\) and \(H(J[\ell])\), completing the proof. 
\end{proof}

\subsubsection{Canonical Intertwining Morphisms}

Given any pair of oriented Lagrangians \(M^\circ, L^\circ\), it is shown in \cite[Thm 2.2.3]{GHQFF09} that there exists canonical isomorphisms for any pair of oriented Lagrangians
	\[T_{M^\circ,L^\circ}\colon C_{L^\circ}\simrightarrow C_{M^\circ}.\]

We describe the construction of \(T_{M^\circ,L^\circ}\) explicitly, following \cite[2.4]{GHQFF09}. 

\begin{construction} \label{con:Intertwining_Morphism}
	Let \(M^\circ= (M, o_M), L^\circ=(L, o_L)\) be two oriented Lagrangians. Set \(I\defeq M\cap L\) and \(n_I \defeq \tfrac{\dim (I^\perp/I)}{2}\).

	The top exterior powers of \(M\) and \(L\) decompose canonically
	\[\begin{aligned}
		\bigwedge^{top}M &= \bigwedge^{top}I \bigotimes \bigwedge^{top}M/I\\ \bigwedge^{top}L &= \bigwedge^{top}I \bigotimes \bigwedge^{top}L/I.
	\end{aligned}\]

	Since the top exterior powers are 1 dimensional, the orientations \(o_M, o_L\) can be decomposed as 
	\[\begin{aligned}
	o_M &= \iota_M\otimes o_{M/I}
	\\o_L &= \iota_L\otimes o_{L/I}
	\end{aligned}\]
	where \(\iota_M, \iota_L \in \bigwedge^{top}I, o_{M/I}\in \bigwedge^{top} M/I\), and \(o_{L/I}\in \bigwedge^{top} L/I\).

	Define the averaging function \(F_{M^\circ,L^\circ}\colon C_{L^\circ}\simrightarrow C_{M^\circ}\) to take a function in \(C_L\) and sum its values over a transversal of \(M\)
	\[F_{M^\circ,L^\circ}[f](v) = \sum_{m+I \in M/I} f(m\cdot v).\]

	Define the normalisation constant
	\[A_{M^\circ,L^\circ} = (G(\tfrac{1}{2},\ell)/\ell)^{n_I} \left(\frac{(-1)^{\binom{n_I}{2}} \frac{\iota_L}{\iota_M}\cdot \omega_\wedge (o_{L/I}, o_{M/I})}{\FF_\ell}\right).\]

	Where: 
	\begin{itemize}
		\item \(\legendre{\cdot}{\FF_\ell}\) is the unique quadratic character of the multiplicative group \(\FF_\ell^\times\). This is the unique non-zero morphism \(\FF_\ell^\times\rightarrow \{\pm 1\}\). (In the case where \(\ell\) is prime, this is the Legendre symbol.)

		\item \(G(\alpha,\ell)\) is the one dimensional Gauss sum
			\[G(\alpha,\ell) = \sum_{z\in \FF_\ell} \psi(\alpha z^2).\]

		\item \(\omega_\wedge\colon\bigwedge^{top} L/I \times \bigwedge^{top} M/I \rightarrow \FF_\ell\) is the pairing induced by the symplectic form \(\omega\).
	\end{itemize}

	Then the canonical intertwining morphism is the averaging morphism times the normalisation constant
	\[T_{M^\circ,L^\circ} =A_{M^\circ,L^\circ} \cdot F_{M^\circ,L^\circ}.\]

\end{construction}

\subsection{Symplectic Actions on the Heisenberg Representations}\label{sec:Symplectic_Action}

Let \(g\in \Sp(V)\), then \(g\) acts on \(H(V)\), which in turn acts on representations of \(H(V)\).

We wish to understand the trace \(\tr(g|\fH(V))\), but by the finite analogue of the Stone Von-Neumann theorem \cite[Theorem 1.1]{GHQuantizationSymplecticVector}, \(\fH(V)\) is isomorphic to \(C_{M^\circ}\) for any given lagrangian \(M^\circ\). Thus it suffices to make the \(g\) action on \(C_{M^\circ}\) explicit.

Suppose that the image of \(M\) under \(g\) is \(gM\), then it is easy to check that there is a map from \(C_{M^\circ}\) to \(C_{gM^\circ}\) via pre-composing by \(g^{-1}\)
\[\begin{aligned}
C_{M^\circ} &\rightarrow C_{gM^\circ}
\\ \phi &\mapsto \phi\circ g^{-1}.
\end{aligned}\]

Thus the action of \(g\) on \(C_{M^\circ}\) is the composite of the following maps (c.f. \cite[Section 2.6]{GH_QuantisationFiniteField})
\[C_{M^\circ} \xrightarrow{g} C_{gM^\circ} \xrightarrow{T_{M^\circ, gM^\circ}} C_{M^\circ}.\]

\begin{lemma}\label{thm:Trace_Any_Lagrangian}
	Let \(g \in \Sp(V)\) be any symplectomorphism. Given any Lagrangian \(M\), and any complement \(M'\) such that \(M\oplus M' = V\), let \(\mathcal S\) be the set
	\[\mathcal S \defeq \{x\in M': gx-x\in M+gM\}.\]
	and for each \(x\in S\), pick \(m_x,n_x\in M\) such that
	\[gx-x=m_x+gn_x.\]
	Then
	\[\tr(g|C_{M^\circ}) = A_{M^\circ\!,gM^\circ} \cdot \sum_{x\in \mathcal S} \psi\left(\frac12\omega(m_x+n_x,x)\right).\]
\end{lemma}

\begin{proof}
	Note that any \(f\in \CC(H(V),\psi)\) is uniquely determined by its values on \(\{0\}\times V \subset H(V)\), since the center must act via the identity. Moreover since the Lagrangian subspace \(M\) acts trivially, \(f\in C_{M^\circ}\) is uniquely determined by its values on the group \(\{0\}\times M'\).

	Thus there is a \(\CC\)-basis \(\{i_x\}_{x\in M'}\) on \(C_{M^\circ}\) indexed by elements in \(M'\) such that \(i_x\) restricts to the indicator function on \(\{0\}\times M'\). Explicitly \(i_x \in C_{M^\circ}\) is the function
	\[\begin{cases}
	i_x(z,x)= \psi(z)& 
	\\i_x(z,m+x)= \psi(z-\tfrac12\omega(m,x))& \text{for } m\in M\text{, since } m\cdot (0,x) = (\tfrac12\omega(m,x),m+x)
	\\i_x(z,v)=0& \text{for } v\notin x+M.
	\end{cases}\]

	Note that via fixing the basis \(\{i_x\}\), the trace of \(g\) on \(C_{M^\circ}\) is simply \(A_{M^\circ\!,gM^\circ}\) times the coefficient of \(i_x\) appearing in \(F_{M^\circ\!,gM^\circ}\circ g[i_x]\). But this coefficient is equal to evaluation at \((0,x)\), so
	\[\tr(g|C_{M^\circ}) = A_{M^\circ\!,gM^\circ} \cdot \sum_{x\in M'}F_{M^\circ\!,gM^\circ}\circ g[i_x](0,x).\]
	We now evaluate the image of each indicator function \(i_x\) under \(F_{M^\circ\!,gM^\circ}\circ g\), at \((0,x)\).

	By precomposing with \(g^{-1}\) we get the function \(g[i_x](z,h)=i_x(z,g^{-1}(h))\), which explicitly evaluates as
	\[\begin{cases}
	g[i_x](z,gx)= \psi(z)& 
	\\g[i_x](z,gm+gx)= \psi(z-\tfrac12\omega(m,x))& \text{for } m \in M
	\\g[i_x](z,v)=0& \text{for } v\notin gx+gM.
	\end{cases}\]

	Finally we put \(g[i_x]\) through the averaging morphism \(F_{M^\circ, gM^\circ}\) to obtain
	\[\begin{aligned}
		F_{M^\circ\!,gM^\circ}\circ g[i_x](0,x)&= \sum_{m+I\in M/I} g[i_x](m\cdot (0,x))
		\\&=\sum_{m+I\in M/I} g[i_x](\tfrac12\omega(m,x),m+x).
	\end{aligned}\]

	Terms in this sum are non-zero only if \(m+x\in gx+gM\), or in other words, the function \(i_x\) contributes to the trace of \(g\) if and only if \(gx-x\in M+gM\). In this case, we write
	\[gx-x=m_x+gn_x\]

	for some fixed \(m_x,n_x\in M\).
	Note further that since the sum is taken over cosets \(m+I\in M/I\), the above sum has precisely one non-zero term. Thus we have
	\[\begin{aligned}
	F_{M^\circ\!,gM^\circ}\circ g[i_x](0,x)&= g[i_x](\tfrac12\omega(m_x,x),m_x+x)
	\\&=g[i_x](\tfrac12\omega(m_x,x),-gn_x+gx)
	\\&=\psi(\tfrac12\omega(m_x,x)-\tfrac12\omega(-n_x,x))
	\\&=\psi(\tfrac12\omega(m_x+n_x,x))
	\end{aligned}\]

	We quickly check that this sum is independent of the choice of \(m_x\) and \(n_x\). If \(m_x,n_x\) were replaced with \(m_x'=m_x+a,n_x'=n_x+b\), then since \(gn_x+m_x=gn_x'+m_x'\) we must have \(a=-gb\in M\cap gM=I\). Then
	\[\omega(a,x)=\omega(a,m_x+x)=\omega(-gb,-gn_x+gx)=-\omega(b,n_x+x)=-\omega(b,x).\]
	Which implies \(\omega(m_x+n_x,x)=\omega(m_x'+n_x',x)\), so this sum is independent of the choices of \(m_x\) and \(n_x\).

	Finally, our trace is
	\[\tr(g|C_{M^\circ}) = A_{M^\circ\!,gM^\circ} \cdot \sum_{x\in \cS} \psi(\tfrac12\omega(m_x+n_x,x))\]
	as desired.
\end{proof}

An immediate and important consequence of the lemma is the following special case:

\begin{corollary}\label{Cor: Invariant Lagrangian}
	Let \(g\in \Sp(V)\) be a semi-simple symplectomorphism. Suppose that there exists a Lagrangian \(M\) that is also invariant under \(g\), i.e. \(g(M)=M\). Then
	\[\tr(g|C_M)=\legendre{\det(g|M)}{\FF_\ell}\sqrt{|V^g|}.\]
	Where \(V^g\) are the \(g\)-fixed points of \(V\). 
\end{corollary}

\begin{proof}
	Since \(g\) is semi-simple, let \(M'\subset V\) be a complementary \(g\)-invariant subspace such that \(V=M\oplus M'\). Then \(M=gM\) and \(gM'=M'\). For any \(x\in \mathcal S\), we have that \(x-gx\in M \cap M' = 0\). Thus \(x\in \mathcal S\) if and only if \(x=gx\). For each such \(x\) we can pick \(n_x=m_x=0\), which turns the trace into the following:
	\[\tr(g|C_{M^\circ}) = A_{M^\circ\!,gM^\circ} \cdot \sum_{x\in \mathcal S} \psi(\tfrac12\omega(0,x))=A_{M^\circ\!,gM^\circ} \cdot \sum_{x\in \mathcal S} 1=  A_{M^\circ\!,gM^\circ} \cdot |\mathcal S|\]
	where 
	\[S= \{x\in M'| gx=x\}= (M')^g\]
	is the set of \(g\)-fixed points of \(M'\). 

	We first compute \(|\mathcal S|\). The symplectic pairing \(\omega\) induces an isomorphism from \(V\) to its dual
	\[\begin{aligned}
		D\colon V&\rightarrow V^*
		\\x &\mapsto \omega(x,-).
	\end{aligned}\]

	Since \(g\) is a symplectomorphism, this isomorphism is \(g\)-equivariant, i.e. for all \(x\in V\)
	\[g\circ D (x)=\omega(x,g^{-1}(-))=\omega(gx,-) = D\circ g(x).\]

	The inclusion \(M\hookrightarrow V\) induces a restriction map \(r\colon V^*\rightarrow M^*\). Composing with \(D\) we note that the kernel of the map \(r\circ D\colon V\rightarrow M^*\) is precisely \(M\). Thus (as \(g\)-modules)
	\[M'\simeq V/M \simeq M^*.\]

	So the fixed points of \(M'\) are the same as the fixed points of \(M^*\). Note that \(\phi\in M^*\) is a fixed point if and only if for any \(m\in M\),  \(\phi(m)=\phi(g^{-1}(m))\). In other words, \(\phi\) must factor through the co-invariant module \(M/(\id-g^{-1})M= M/(g-\id)M=:M_{g}\), and thus
	\[(M^*)^{g}\simeq (M_{g})^*.\]
	
	Moreover, from the exact sequence
	\[0\rightarrow M^{g} \rightarrow M \xrightarrow{g-\id} M\rightarrow M_{g}\rightarrow 0,\]
	we can deduce that (since \(M\) is a finite module) \(M\) and \(M'\) have the same number of invariant elements
	\[|(M')^{g}|=|(M^*)^{g}|=|(M_{g})^*|=|M^{g}|.\]

	Thus we conclude
	\[|(M')^{g}|=\sqrt{|(M')^{g}||M^{g}|}=\sqrt{|(M'\times M)^{g}|}=\sqrt{|V^{g}|}.\]

	Finally we compute \(A_{gM^\circ,M^\circ}\). In our setting \(I\defeq M\cap gM=M\) and \(n_I\defeq\frac{\dim(I^\perp/I)}{2}=0\). Moreover the \(\omega_{\wedge}\) term vanishes since \(M/I\) is trivial. Thus we have
	\[A_{gM^\circ,M^\circ}=\legendre{\frac{o_{M}}{o_{gM}}}{\FF_\ell}= \legendre{\det(g|M)}{\FF_\ell}.\]

	Thus \(\tr(g|C_M)=\legendre{\det(g|M)}{\FF_\ell}\sqrt{|V^g|}\) as desired.
\end{proof}

\subsection{Invariant Spaces of a Symplectic Vector Space}\label{sec:Symplectic_Linear_Algebra}

We prove some facts about the invariant spaces of a symplectomorphism \(g\in \Sp(V)\). Let $\omega$ be a symplectic form on a \(2g\)-dimensional vector space $V$ over a finite field $\FF$ of characteristic \(\neq 2\). Let $g \in \Sp(V,\omega)$ be a symplectomorphism, and suppose \(g\) is represented by the matrix \(S\) under the standard symplectic basis. 

\begin{proposition}
	\(\det(S)=1\)
\end{proposition}

\begin{proof}
	Viewing \(\omega\) as an element of \(\Lambda^2V^*\), consider \(\Lambda^g \omega \in \Lambda^{2g}V^*\). The top exterior power is 1 dimensional and the determinant of \(S\) is equal to its scalar action on the top exterior power \(\Lambda^{2g}V^*\). But since \(S^*\omega = \omega\), it follows that \(S^*(\Lambda^g \omega)= \Lambda^g \omega\) and thus \(\det S = 1\).
\end{proof}

\begin{proposition}\label{prop:eval_multiplicity}
	Let \(\chi_S(t)= \det(tI-S)\) denote the characteristic polynomial of \(S\). If \(\lambda\neq \pm1\) is a root of \(\chi_S\) with multiplicity \(d\), then \(\lambda^{-1}\) is also a root with multiplicity \(d\). 

	Moreover, if \(\pm1\) is a root of \(f_S\), then it will occur with even multiplicity.
\end{proposition}

\begin{proof}
	Let \(\Omega\) be the matrix representing \(\omega\), then \(S\) satisfies \(S^T \Omega S = \Omega\). Rearranging we get \[S^{-1}=\Omega^{-1}S^T\Omega\]
	Thus \(S^{-1}\) and \(S\) have equal characteristic polynomials. On the other hand,
	\[\chi_{S^{-1}}(t)= \det(tI-S^{-1})=\det(tS-I)\det(S^{-1})=t^{2g}\det(S-\tfrac{1}{t}I)= t^{2g}\chi_S(\tfrac{1}{t}).\]

	Thus if \(\lambda\) is a root of \(\chi_S\) with multiplicity \(d\), then so is \(\lambda^{-1}\). 

	The product of all the roots of \(\chi_S(t)\) must equal \(\det(S)=1\). Since every non-\(\pm1\) eigenvalue must come in reciprocal pairs, it follows that if \(-1\) is a root of \(\chi_S\), it must occur with even multiplicity. 
	
	Finally, \(S\) is even-dimensional and thus \(1\) must also occur as a root with even multiplicity. 
\end{proof}

For an eigenvalue \(\lambda\) of \(S\) with algebraic multiplicity \(d\), let \(S_\lambda \defeq \ker(\lambda I -S)^d \) denote the \textit{generalised eigenspace} of \(V\) corresponding to \(\lambda\).

\begin{proposition}
	Suppose that the characteristic polynomial of \(S\) splits completely in \(\FF\). Given (not neccesarily distinct) eigenvalues \(\lambda, \mu\) of \(S\), then
	\begin{enumerate}
		\item If \(\lambda\mu\neq 1\), then the spaces \(S_\lambda\) and \(S_\mu\) are orthogonal. i.e. for \(v\in S_\lambda, w\in S_\mu\),
		\[\omega(v,w)=0.\]
		In particular this implies that when \(\lambda\neq\pm1\), \(S_\lambda\) is a isotropic subspace of \(V\). 
		\item If \(\lambda = \pm1\), then \(S_\lambda\) is a symplectic subspace of \(V\). 
	\end{enumerate}
\end{proposition}

\begin{proof}
	1) Let \(v\in S_\lambda, w\in S_\mu\). We say \(v\) has rank \(r\) if it lies in the kernel of \((\lambda I -S)^r\) but not \((\lambda I -S)^{r-1}\), and similarly for \(w\), we will induct on the sum of the ranks of \(v\) and \(w\).

	The base case is when \(v,w\) are both eigenvectors, in that case
	\[\omega(v,w)=\omega(Sv,Sw)= \omega(\lambda v,\mu w)= \lambda\mu\omega(v,w).\]
	Since \(\lambda\mu\neq 1\), it follows that \(\omega(v,w)=0\).

	For the inductive case, suppose \(v\) has rank \(r_1\) and \(w\) has rank \(r_2\), and assume that \(\omega(v',w')=0\) whenever the ranks of \(v',w'\) sum to less than \(r_1+ r_2\). 
	
	Then by the definition of rank, \(Sv=\lambda v+ v'\) and \(Sw=\mu w+ w'\), where \(v',w'\) have ranks \(r_1-1, r_2-1\) respectively (If a vector has rank 0 then it is simply the zero vector). Then
	\[\begin{aligned}
		\omega(v,w)=\omega(Sv,Sw)=\omega(\lambda v+v',\mu w+ w')&= \omega(v',w')+\omega(\lambda v,w')+\omega(v',\mu w)+\omega(\lambda v,\mu w)
		\\&=\omega(\lambda v,\mu w)=\lambda\mu\omega(v,w).
	\end{aligned}\]

	Once again since \(\lambda\mu\neq 1\), it follows that \(\omega(v,w)=0\), completing the proof. 

	2) To show that \(S_\lambda\) is a sympletic subspace of \(V\), it suffices to show that \(\omega|_{S_\lambda}\) is a non-degenerate form, since bilinearity and antisymmetry is inherited from \(\omega\). 
	
	Suppose that \(v_0\in S_\lambda\) satisfies \(\omega(v_0,w)=0 \) for any \(w\in S_\lambda\). Note that since \(\lambda=\pm1\), for any \(\mu\neq\lambda\), \(\mu\lambda\neq 1\) and thus by the previous part, \(\omega(v_0,w)=0\) for all \(w\in S_\mu\). Since \(V\) is a direct sum of all its generalised eigenspaces, this implies that \(\omega(v_0,w)=0\) for all \(w\in V\). Since \(\omega\) is non-degenerate on \(V\), it follows that \(v=0\), proving non-degeneracy of \(\omega\) in \(S_\lambda\). 
\end{proof}

\begin{lemma}\label{thm:V_Decomposition}
	Suppose \(g\in \Sp(V)\) is semisimple, then it is possible to decompose \(V\) as a direct sum of \(g\)-invariant symplectic subspaces
	\[V=S_1\oplus S_{-1}\oplus V_1\oplus \cdots \oplus V_r.\]
	Where \(S_{\pm1}\) is the \(\pm 1\) eigenspace of \(g\) respectively, and \(V_i\) are subspaces such that the restriction \(g|_{V_i}\) of the \(g\)-action to \(V_i\) has characteristic polynomial \(f_i\), where \(f_i\) has no repeated roots. 
\end{lemma}

\begin{proof}
	For an irreducible monic polynomial \(h(t)\), let \(h^-(t)\defeq t^{\deg h} h(1/t)\) denote the irreducible polynomial with roots being the reciprocals of those of \(h(t)\). 

	Since roots of \(f\) come in pairs of \(\lambda,\lambda^{-1}\) with equal multiplicity, then for each irreducible factor \(h(t)\) of \(f(t)\), either \(h^-(t)=h(t)\), or \(h^-(t)\) must also be an irreducible factor of \(f(t)\) with the same multiplicity as \(h(t)\). 

	Thus we can write \[f(t)=(t+1)^a(t-1)^bf_1(t)^{k_1}f_2(t)^{k_2}\dots f_s(t)^{k_s}.\]
	Where each \(f_i\) is either an irreducible factor satisfying \(f_i(t)=f_i^-(t)\), or \(f_i\) is a product of two irreducible factors of the form \(h(t)h^-(t)\). In both cases \(f_i(t)\) has no repeated roots, since finite extensions of \(\FF\) are separable.
	Thus we can decompose \(V\) into \(g\)-invariant subspaces \[V=S_1\oplus S_{-1}\oplus W_1\oplus \cdots \oplus W_s\]

	where \(W_i\) is the vector subspace of \(V\) corresponding to \(f_i(t)^{k_i}\). Then over the algebraic closure \(W_i\otimes \bar{\FF}\) is a symplectic subspace of \(V\otimes \bar{\FF}\) since it is a direct sum of \(S_{\lambda}\oplus S_{\lambda^{-1}}\) for every pair of roots \(\lambda,\lambda^{-1}\) of \(f\). Thus \(W_i\) is also a symplectic subspace of \(V\). 

	Thus it suffices to show that each \(W_i\) can be further decomposed into \(W_{i1}\oplus \cdots \oplus W_{il}\) where each \(W_{ij}\) is a \(g\)-invariant symplectic subspaces of \(W_i\). 
	We know that the restricted action \(g|_{W_i}\) has characteristic polynomial \(f_i^{k_i}\).
	Let \(T\) be a matrix with characteristic polynomial \(f_i\) that is also symplectic with respect to a symplectic form \(\varpi\), then the block diagonal matrix
	\[\begin{pmatrix}
		T\\ &T \\ && \ddots \\ &&&T
	\end{pmatrix}\]
	with \(k_i\) diagonal blocks will have characteristic polynomial \(f_i^{k_i}\). This matrix is also symplectic with respect to the symplectic form \(\varpi^{\otimes k_i}\). 

	Applying a suitable change of basis, we can obtain a matrix \(U\) that is symplectic with respect to the symplectic form on \(W_i\). 

	Finally, the summary in \cite[p.7]{WallSemisimpleConjugacyClasses} states that any two semisimple elements in \(\Sp(V)\) which are conjugate over \(\GL(V)\) are also conjugate over \(\Sp(V)\). Since both \(U\) and \(g|_{W_i}\) are semisimple with the same characteristic polynomial, this means that under a suitable basis \(g|_{W_i}\) has matrix \(U\). It is clear from construction that the matrix \(U\) decomposes into \(g\)-invariant symplectic subspaces each with characteristic polynomial \(f_i\), so we are done.
\end{proof}

\subsection{Evaluating the Trace}\label{sec:Trace_Sum_Calculation}

\begin{lemma} \label{lem:Trace_sum_of_spaces}
	Let \(g\in \Sp(V)\) be a symplectomorphism. Suppose \(V\) decomposes into a direct sum of \(g\)-invariant symplectic subspaces	\[V=V_1\oplus V_2\oplus\cdots\oplus V_r\]
	Let \(\mathfrak H_i= \mathfrak H(V_i)\) be the unique representation of \(V_i\) with identity central character as described in \cref{sec:Heisenberg_Rep}. Then
	\[\tr(g|\mathfrak H(V))=\prod_{i=1}^r\tr(g|\mathfrak H_i).\]
\end{lemma}

\begin{proof}
	This follows directly from \cite[Prop 2.14]{GHQuantizationSymplecticVector} which asserts that
	\[\mathfrak H(V)=\bigotimes_{i=0}^r \mathfrak H_i\]
	and the fact that the trace on a tensor product of vector spaces is equal to the product of the traces on each individual component. 
\end{proof}

\begin{lemma}\label{thm:trace_separable_char_poly}
	Suppose that \(g\in \Sp(V)\) is a semi-simple symplectomorphism which has characteristic polynomial \(f(t)\) with no repeated roots, then
	\[\tr(g|\fH(V))=\legen{(-1)^{n}f(1)}{\FF_\ell},\]
	where \(\dim V=2n\). 
\end{lemma}

\begin{proof}
	First we note that since we have assumed that \(f(t)\) has no repeated roots, this also means that \(g\) does not have any eigenvalues equal to \(\pm1\). This is because by \cref{prop:eval_multiplicity} any eigenvalue of \(\pm1\) must occur with even multiplicity. 

	Fix a standard symplectic basis of \(V\) and let \(S\) denote the matrix representing \(g\) in this basis. Since \(f\) has no repeated roots, this means that any matrix in \(T\in\GL(V)\) with the same characteristic polynomial \(f\) will be similar to \(S\) (since neither \(S\) nor \(T\) will have any Jordan blocks as every eigenvalue has multiplicity 1). 

	Furthermore, by the summary in \cite[p. 7]{WallSemisimpleConjugacyClasses}, it is shown that the \(\GL(V)\) conjugacy class of \(S\) intersects \(\Sp(V)\) at a unique conjugacy class. This means that any symplectic matrix \(T\in \Sp(V)\) with the same characteristic polynomial as \(g\) will be similar to \(S\) via a symplectomorphism.

	Now suppose \(g\) has characteristic polynomial
	\[f(t)=t^{2n}+a_1t^{2n-1}+\cdots+a_{n-1}t^{n+1} +a_{n}t^{n}+a_{n-1}t^{n-1}+\cdots+a_1t+1\]
	and construct the block matrix
	\[S=\begin{pmatrix}
		A&B\\0&C
	\end{pmatrix}\]

	Where \(A\) and \(B\) have dimension \(n+1\times n\), \(C\) has dimension \(n-1\times n\) and they are

	\[A=\begin{pmatrix}
		0&0&0&\cdots
		\\1&0&0&\cdots
		\\0&1&0&\cdots
		\\0&0&1&\cdots
		\\\vdots&\vdots&\vdots&\ddots
	\end{pmatrix},
	B=\begin{pmatrix}
		0&0&0&\cdots& -1
		\\0&0&0&\cdots& -a_1
		\\0&0&0&\cdots& -a_2
		\\\vdots&\vdots&\vdots&\vdots&\vdots
		\\0&0&0&\cdots& -a_{n-1}
		\\-a_1&-a_2&-a_3&\cdots& -a_n
	\end{pmatrix},
	C=\begin{pmatrix}
		1&0&0&0&\cdots
		\\0&1&0&0&\cdots
		\\0&0&1&0&\cdots
		\\\vdots&\vdots&\vdots&\ddots
	\end{pmatrix}.\]

	It can be checked by a straightforward computation that \(\chi_S(t)=f(t)\) and that \(S\) preserves the standard symplectic form. Thus under suitable symplectic change of basis, \(g\) is represented by the matrix \(S\). 

	Let \(v_1,\dots,v_n,w_1,\dots,w_n\) be this standard symplectic basis, then explicitly \(g\) is the linear map
	\[g\colon\begin{cases}
		v_i \mapsto v_{i+1} & i\neq n
		\\v_n \mapsto w_1
		\\w_i \mapsto -a_iw_1+w_{i+1}& i\neq n
		\\w_n \mapsto -(v_1+a_1v_2+a_2v_3+\cdots+a_{n-1}v_n+a_nw_1).
	\end{cases}\]

	Now that we have a symplectic basis of \(g\), we fix the Lagrangian \(M\) to be the span of \(v_1,\dots,v_n\), the complement \(M'\) to be the span of \(w_1,\dots,w_n\) and we apply \cref{thm:Trace_Any_Lagrangian} to calculate the trace of \(g\). 

	We first determine the set \(\cS = \{x\in M'| gx-x\in M+gM\}\). Noting that \(M+gM=M+\FF_\ell w_1\), we let \(x=\alpha_1w_1+\dots+\alpha_n w_n\) and we require that \(gx-x\) have no \(w_2\cdots w_n\) terms. For each \(2\leq i\leq n\), the coefficient of \(w_i\) in \(x-gx\) being zero implies that \(\alpha_i=\alpha_{i-1}\). Thus \(\cS\) is the 1 dimensional subspace spanned by \(w_1+w_2+\dots+w_n\). 

	Letting \(x=w_1+w_2+\dots+w_n\), we compute that
	\[\begin{aligned}
		gx-x&=-(v_1+a_1v_2+a_2v_3+\cdots+a_{n-1}v_n)-(a_1+a_2+\cdots+a_n+1)w_1
		\\&=-(v_1+a_1v_2+a_2v_3+\cdots+a_{n-1}v_n)-g[(a_1+a_2+\cdots+a_n+1)v_n] 
	\end{aligned}.\]

	Thus we can pick \(m_x=-(v_1+a_1v_2+a_2v_3+\cdots+a_{n-1}v_n)\) and \(n_x=-(a_1+a_2+\cdots+a_n+1)v_n\). Then we see that
	\[\begin{aligned}
		\omega(m_x+n_x,x)&=-(a_n+2a_{n-1}+2a_{n-2}+\cdots+2a_{1}+2)
		\\&=-f(1)
	\end{aligned}.\]
	Where \(f(1)\) is the sum of the coefficients of the characteristic polynomial of \(g\). 

	Thus we can compute the sum
	\[\begin{aligned}
	\sum_{x\in \cS}\psi\left(\frac12\omega(m_x+n_x,x)\right)&=\sum_{x\in \cS}\psi\left(\frac12\omega(m_x+n_x,x)\right)
	\\&=\sum_{k\in \FF_\ell}\psi\left(\frac{-f(1)}{2}k^2\right)
	\\&=\legen{-\frac{1}{2}f(1)}{\FF_\ell}G(1,\ell)
	\end{aligned}.\]
	Where the last equality is from \cite[p.85]{LangAlgebraicNumberTheory}

	Finally we compute \(A_{M^\circ\!,gM^\circ}\). We fix the orientation 
	\[o_{M}=v_1\wedge v_2\wedge\dots\wedge v_n\] and so 
	\[o_{gM}=g(o_M)=v_2\wedge \cdots\wedge v_n\wedge w_1=(-1)^{n-1}w_1\wedge v_2\wedge \cdots\wedge v_n.\]

	Then we can decompose the orientations via \(\iota_{M}=\iota_{gM}=v_2\wedge \cdots\wedge v_n\), and \(o_{M/I}=v_1\), \(o_{gM/I}=(-1)^{n-1}w_1\). We note that \(n_I=1\) and so
	\[\begin{aligned}
		A_{M^\circ,gM^\circ} &= (G(\tfrac{1}{2},\ell)/\ell)^{n_I} \left(\frac{(-1)^{\binom{n_I}{2}} \frac{\iota_{gM}}{\iota_M}\cdot \omega_\wedge (o_{gM/I}, o_{M/I})}{\FF_\ell}\right)
		\\&=(G(\tfrac{1}{2},\ell)/\ell)\left(\frac{(-1)^{n-1}\omega(w_1,v_1)}{\FF_\ell}\right)
		\\&=\frac{1}{\ell}G(\tfrac12,\ell)\legen{(-1)^n}{\FF_\ell}
		\\&=\frac{1}{\ell}G(1,\ell)\legen{(-1)^n\frac12}{\FF_\ell}
	\end{aligned}.\]

	Thus finally by \cref{thm:Trace_Any_Lagrangian} we have that
	\[\begin{aligned}
	 \tr(g|\fH(V))=\tr(g|C_{M^\circ})&=A_{M^\circ\!,gM^\circ} \cdot \sum_{x\in \mathcal S} \psi\left(\frac12\omega(m_x+n_x,x)\right)
	 \\&=\frac{1}{\ell}G(1,\ell)\legen{-\frac12}{\FF_\ell} \cdot \legen{(-1)^n\frac{1}{2}f(1)}{\FF_\ell}G(1,\ell)
	 \\&=\frac{1}{\ell}G(1,\ell)^2\legen{(-1)^{n-1}f(1)}{\FF_\ell}
	 \\&=\legen{-1}{\FF_\ell}\legen{(-1)^{n-1}f(1)}{\FF_\ell}
	 \\&=\legen{(-1)^{n}f(1)}{\FF_\ell}.
	\end{aligned}\]

	Where the second last equality is due to \cite[p. 87]{LangAlgebraicNumberTheory}, which asserts that
	\[G(1,\ell)= \begin{cases}
		\sqrt{\ell} & \ell\equiv 1\pmod 4
		\\i\sqrt{\ell}& \ell\equiv 3\pmod 4.
	\end{cases}\]
\end{proof}

\begin{theorem} \label{thm:Trace_Char_Poly}
	Suppose \(g\in \Sp(V)\) is semisimple with minimal polynomial \(f\) and \(\dim V=2n\). Then
	\[\tr(g|\fH(V))=\legen{(-1)^{n-(\dim V^g/2)}\det'(1-g)}{\FF_\ell}\sqrt{|V^g|},\]
	where \(\det'\) is the \textit{regularised determinant}.
\end{theorem}
	
\begin{proof}
	By \cref{thm:V_Decomposition} decompose \(V\) as a direct sum of \(g\)-invariant symplectic subspaces
	\[V=S_1\oplus S_{-1}\oplus V_1\oplus \cdots \oplus V_r.\]
	Where \(S_{\pm1}\) is the \(\pm 1\) eigenspace of \(g\) respectively, and \(V_i\) are subspaces such that the restriction \(g|_{V_i}\) of the \(g\)-action to \(V_i\) has characteristic polynomial \(f_i\), where \(f_i\) has no repeated roots. 

	Then by \cref{lem:Trace_sum_of_spaces} we can write
	\[\tr(g|\fH(V))=\tr(g|\mathfrak H(S_1))\tr(g|\mathfrak H(S_{-1}))\prod_{i=1}^r\tr(g|\mathfrak H(V_i)).\]

	Since \(g\) acts as \(\Id\) and \(-\Id\) on \(S_1\) and \(S_{-1}\) respectively, it is easy to compute the traces using \cref{Cor: Invariant Lagrangian} by taking any invariant Lagrangian
	\[\tr(g|\mathfrak H(S_1))=\sqrt{|S_1|}=\sqrt{|V^g|},\]
	and
	\[\tr(g|\mathfrak H(S_{-1}))=\legen{(-1)^{\dim S_{-1}/2}}{\FF_\ell}.\]
	
	In particular, note that since \(S_{-1}\) is even dimensional, the characteristic polynomial of \(g\) restricted to \(S_{-1}\) is an even power of \((t+1)\), which in particular means that if we denote \(f_{-1}\) as the characteristic polynomial of \(g\) on \(S_{-1}\), then we can write instead
	
	\[\tr(g|\mathfrak H(S_{-1}))=\legen{(-1)^{\dim S_{-1}/2}f_{-1}(1)}{\FF_\ell}.\]

	Combining this with \cref{thm:trace_separable_char_poly} we obtain
	\[\begin{aligned}
		\tr(g|\fH(V))&=\sqrt{|V^g|}\cdot \legen{(-1)^{\dim S_{-1}/2}f_{-1}(1)}{\FF_\ell}\cdot\prod_{i=1}^r\legen{(-1)^{(\dim V_i)/2}f_i(1)}{\FF_\ell}
		\\&=\sqrt{|V^g|}\legen{(-1)^{(\dim S_{-1}+\dim V_1+\cdots \dim V_r)/2}f_{-1}f_1f_2\dots f_r(1)}{\FF_\ell}
	\end{aligned}.\]

	Note that \(\dim S_{-1}+\dim V_1+\cdots \dim V_r =2n-\dim S_1=2n-\dim V^g\), since the 1-eigenspace \(S_1\) is precisely the \(g\) fixed points \(V^g\). 

	Finally, we note that the kernel of \(1-g\) is precisely the 1-eigenspace \(S_1\), so the characteristic polynomial of \(1-g\) on the quotient \(V/S_1\) is precisely \(f_{-1}f_1f_2\dots f_r(1+t)\). Thus we have the regularised determinant \(\det'(1-g)\) is equal to \(f_{-1}f_1f_2\dots f_r(1)\) and so we have 
	\[\tr(g|\fH(V))=\legen{(-1)^{n-(\dim V^g/2)}\det'(1-g)}{\FF_\ell}\sqrt{|V^g|},\]

	as desired.
\end{proof}

Finally we prove the main theorem of this section.

\begin{proof}[Proof of \cref{thm:TraceFormula}]
	By \cref{thm:UniqueRep} and the Stone-von Neumann property, we know that \(\cH\) is isomorphic to \(\fH(J[\ell])\). Thus
	\[\tr(\Fr_q | \cH) = \tr(\Fr_q|\fH(J[\ell])).\]
	Then applying \cref{thm:Trace_Char_Poly} and noting that the \(\Fr_q\) fixed points of \(J[\ell]\) is precisely \(J[\ell](\FF_q)\) we obtain our desired result.
\end{proof}

\section{An Arithmetic Path Integral} \label{sec:Arithmetic_PathIntegrals}
\subsection{Defining an Arithmetic Action} \label{sec:Arithmetic_PathIntegrals_Def}

As mentioned in the introduction, for \(J\) the Jacobian of \(X\), there is a diagram
\[\begin{tikzcd}
	{\bar{J}[\ell]} & J[\ell] \\
	{\Spec(\bar{\FF}_q)} & {\Spec(\FF_q).}
	\arrow[hook, from=1-1, to=1-2]
	\arrow[from=1-1, to=2-1]
	\arrow[from=1-2, to=2-2]
	\arrow[from=2-1, to=2-2]
\end{tikzcd}\]
We can consider the finite group scheme of \(\ell\)-torsion points of the Jacobian, \(J[\ell]\), to be a manifold fibered over the circle \(\Spec(\FF_q)\) with finite fibres.

We can then view a rational \(\ell\)-torsion point \(\gamma\in J[\ell](\FF_q)\) as a section over this `fibre bundle over \(S^1\)'
\[\begin{tikzcd}
	{J[\ell]} \\
	{\Spec(\FF_q).}
	\arrow[shift right=2, from=1-1, to=2-1]
	\arrow["\gamma"', shift right=2, from=2-1, to=1-1]
\end{tikzcd}\]
We take \(J[\ell]\) to be the arithmetic analogue of the phase space. Our arithmetic pairing will be of the form
\[A\colon J[\ell](\FF_q)\times J[\ell](\FF_q) \rightarrow \frac{1}{\ell}\ZZZZ,\]
with \(A(\gamma)\defeq A(\gamma,\gamma)\).

As a sum over the space \(J[\ell](\FF_q)\), this can be viewed as being analogous to the space of all loops in the phase space. However, the integral described in \cref{sec:PhysicsBackground} is over not just closed loops, but paths where just the position co-ordinates are periodic. It seems our arithmetic result is an analogue of the case where the integral over the phase space loops is equal to the entire integral. We are not entirely sure when this holds, but it seems related to the exactness of the semi-classical approximation. 

Let \(\CH_0(X)\simeq \Pic(X)\) denote the Chow group, then by geometric class field theory there is a reciprocity map
\[\operatorname{Rec}\colon\CH_0(X) \rightarrow \pi_1^{\ab}(X).\]

Let \(\CH_0(X)^0\simeq \Pic^0(X)\simeq J(\FF_q)\) denote the subgroup of \(\CH_0\) consisting of degree 0 algebraic cycles. Then via the commutative diagram

\[\begin{tikzcd}[cramped]
	0 & {\CH_0(X)^0} & {\CH_0} & \ZZ & 0 \\
	0 & {\pi_1^{\ab}(X)^0=\pi_1^{\ab}(\bar{X})} & {\pi_1^{\ab}(X)} & {\hat\ZZ}\simeq \Gal(\bar{\FF}_q/\FF_q) & 0
	\arrow[from=1-1, to=1-2]
	\arrow[from=1-2, to=1-3]
	\arrow[from=1-2, to=2-2]
	\arrow["\deg", from=1-3, to=1-4]
	\arrow["{\operatorname{Rec}}"', from=1-3, to=2-3]
	\arrow[from=1-4, to=1-5]
	\arrow[from=1-4, to=2-4]
	\arrow[from=2-1, to=2-2]
	\arrow[from=2-2, to=2-3]
	\arrow[from=2-3, to=2-4]
	\arrow[from=2-4, to=2-5]
\end{tikzcd}.\]

The reciprocity map restricts to
\[\operatorname{Rec}\colon J(\FF_q)=\CH_0(X)^0 \rightarrow \pi_1^{\ab}(X)^0,\]
and also to the \(\ell\)-torsion of both groups
\[\operatorname{Rec}\colon J[\ell](\FF_q)=\CH_0(X)^0[\ell] \rightarrow \pi_1^{\ab}(X)^0[\ell].\]

On the other hand, an \(\ell\)-torsion point \(\gamma\in J[\ell](\FF_q)\) defines a line bundle \(L_\gamma\) over \(X\) such that \((L_\gamma)^{\otimes\ell} \simeq \mathcal O_X\).

Suppose we fix an isomorphism \(f\colon (L_\gamma)^{\otimes\ell}\simto \OX\). Since we assumed that \(\mu_\ell\subset \FF_q\), fix an isomorphism between \(\mu_\ell\) and \(\fZZZZ{\ell}\) and identify the two. 
Then we can define a \(\fZZZZ{\ell}=\mu_\ell\) torsor \(c_{\gamma,f}\) via the construction
\[c_{\gamma,f}(U) \defeq \left\{y\in \Gamma(L_\gamma,U) : f(y^{\otimes \ell})= 1\right\}\]

for any \'etale map \(U\rightarrow X\). This torsor defines a class in \(c_{\gamma,f}\in H^1\left(X, \frac{1}{\ell}\ZZZZ \right)\). 

\begin{lemma} \label{thm:TorsorAddition}
	Given two torsion points \(\beta,\gamma \in J[\ell](\FF_q)\) and isomorphisms \(g\colon (L_\gamma)^{\otimes \ell} \simrightarrow \OX\), \(f\colon (L_\beta)^{\otimes \ell} \simrightarrow \OX\), be chosen isomorphisms of line bundles respectively. Then \(L_{\beta+\gamma}=L_\beta\otimes L_\gamma\), and as classes in \(H^1\left(X, \frac{1}{\ell}\ZZZZ \right)\) the addition of torsors is given by
	\[c_{\beta,f}+c_{\gamma,g}= c_{\beta+\gamma, g\otimes f}.\]
\end{lemma}

\begin{proof}
	\(L_{\beta+\gamma}=L_\beta\otimes L_\gamma\) follows from the fact that the group law on the Jacobian is the same as the group law on the Picard group, which is the tensor product of line bundles. 

	In order to add the two torsors \(c_{\beta,g}\) and \(c_{\gamma,f}\), we first take the product sheaf \(c_{\beta,f}\times c_{\gamma,f}\) which is a \(\fZZZZ{\ell}\times\fZZZZ{\ell}\) torsor, then we pushout along the summation map \(\fZZZZ{\ell}\times\fZZZZ{\ell} \xrightarrow{+} \fZZZZ{\ell}\) to obtain the torsor \(c_{\beta,g}+c_{\gamma,f}\). For any \'etale \(U\to X\), this pushout identifies \((y,z)\in c_{\beta,g}\times c_{\gamma,f}(U)\) with \((\zeta y,\zeta^{-1}z)\), for any root of unity \(\zeta\in \mu_\ell\). 

	This sheaf is clearly the same as the torsor \(c_{\beta+\gamma, g\otimes f}\), and thus we conclude that \(c_{\beta,g}+c_{\gamma,f}= c_{\beta+\gamma, g\otimes f}\).
\end{proof}

\begin{corollary} \label{cor:Jac_to_H1}
	Any two isomorphisms \(f,f'\colon (L_\gamma)^{\otimes\ell} \simrightarrow  \OX\) will differ only by the scaling of a constant \(a\in \FF_q^\times\), and thus the torsors \(c_{\gamma,f}\) and \(c_{\gamma,f'}\) will differ by an element in \(H^1\left(\Spec(\FF_q), \frac{1}{\ell} \ZZZZ \right)\), and so an \(\ell\)-torsion point \(\gamma\) defines a class in
	\[c_\gamma \in \frac{H^1\left(X, \frac{1}{\ell}\ZZZZ \right)}{H^1\left(\Spec(\FF_q), \frac{1}{\ell} \ZZZZ \right)}.\]		
\end{corollary}

\begin{proof}
	This follows directly from the previous lemma by taking \(\beta=0\) and \(g=a\) being the scaling map. 
\end{proof}

\begin{remark} \label{rmk:Kummer_iso}
Note that by the Kummer exact sequence there is an exact sequence of cohomology
\[0\rightarrow\FF_q^\times/\FF_q^{\times \ell}\xto{\kappa} H^1(X,\mu_\ell) \rightarrow H^1(X,\Gm)[\ell]\rightarrow 0,\]
and \(H^1(X,\Gm)[\ell]= \Pic(X)[\ell]=\Pic^0(X)[\ell]=J[\ell](\FF_q)\), and thus there is a surjective map given by pushing out a \(\mu_\ell\) torsor along the map \(\mu_\ell\to\Gm\) to obtain a \(\Gm\) torsor
\[s\colon H^1(X,\mu_\ell)\twoheadrightarrow J[\ell](\FF_q).\]

Quotienting by the kernel, the inverse of the map
\[\bar{s}\colon\frac{H^1(X,\mu_\ell)}{\im(\kappa)} \simrightarrow J[\ell](\FF_q)\]
is precisely the map \(\gamma\mapsto c_\gamma\) defined above.
\end{remark}

Finally, since there is an isomorphism
\[H^1\left(X, \tfrac{1}{\ell}\ZZZZ \right) = \Hom\left(\pi^{\ab}_1(X), \tfrac{1}{\ell}\ZZZZ \right)\]
we define the pairing

\begin{definition}\label{def:A_Pairing}
	Define \(A\) to be the pairing of elements in \(\CH_0(X)^0[\ell]=J[\ell](\FF_q)\) given by
	\[A(\gamma, \beta) \defeq c_\gamma(\operatorname{Rec}(\beta)),\]
	where we view the torsor \(c_\gamma\) as a homomorphism from \(\pi^{\ab}_1(X)\) to \(\fZZZZ{\ell} \). Noting that the image of the reciprocity map lies inside the kernel of \(\pi_1(X)\rightarrow \Gal(\bar{\FF}_q/\FF_q)\). Thus this is a well-defined function
	\[A\colon J[\ell](\FF_q)\times J[\ell](\FF_q) \rightarrow \fZZZZ{\ell}.\]
\end{definition}

\begin{proposition}\label{prop:A_nondegen_bilinear}
	Suppose \(J(\FF_q)\) has no points of order \(\ell^2\), then the pairing \(A\) is non-degenerate bilinear form on the \(\FF_\ell\) vector space \(J[\ell](\FF_q)\).
\end{proposition}
\begin{proof}
	The linearity of the first argument follows from \cref{thm:TorsorAddition}. 
	The linearity of the second argument follows from the linearity of the reciprocity map and \(\Hom\). 

	Fix a non-zero \(\beta\in J[\ell](\FF_q)\), since \(J(\FF_q)\) has no points of order \(\ell^2\), it follows that \(\beta\) is not \(\ell\)-divisible. Since the reciprocity map \(\Rec\colon J(\FF_q)\to \pi_1^\ab(X)^0\) is an isomorphism, \(\Rec(\beta)\) is an \(\ell\)-torsion point of \(\pi_1^\ab(X)\) that is also not \(\ell\)-divisible. 
	There must exist a homomorphism \(\phi\in \Hom(\pi_1^\ab(X), \tZZZZ{\ell})\) such that \(\phi(\Rec(\beta))\neq 0\). Since \(\Hom(\pi_1^\ab(X), \tZZZZ{\ell})=H^1(\pi_1^\ab(X), \tZZZZ{\ell})\) surjects onto \(J[\ell](\FF_q)\), there must exist a \(\gamma\) such that \(c_\gamma(\Rec(\beta))\neq 0\), so \(A\) is non-degenerate. 
\end{proof}

Letting \(A(\gamma) = A(\gamma,\gamma)\), we define the \emph{arithmetic path integral} of \(A\) to be
\[\sum_{\gamma \in J[\ell](\FF_q)} e^{2\pi i A(\gamma)}.\]

Before we compute this path integral, let us first deduce some properties of the action \(A\). 

\subsection{Relation to the abelian Chern-Simons Pairing}\label{sec:Abelian_CS}

We show that the action \(A\) above can be identified with a function field analogue of the Abelian Chern-Simons pairing defined in \cite{CKKPPYAbelianArithmeticChern}. Throughout this section we assume that \(\mu_{\ell^2} \subset \FF_q\), and thus by fixing an isomorphism \(\zeta\colon \mu_{\ell^2} \simto \Zmod{\ell^2}\) as Galois modules and sheaves, we can identify all cohomology groups with \(\mu_{\ell^2}\) and \(\Zmod{\ell^2}\) coefficients, as well as cohomology groups with \(\mu_\ell\) and \(\ZlZ\) coefficients.

We recall the statement of Artin--Verdier Duality applied to the sheaf  \(\mathcal F= \ZlZ\cong \mu_\ell\). 

\begin{theorem}[{\cite[Cor. 3.3]{MilneDuality}}]
	There is a pairing given by cup product
	\[\anglebrackets{\cdot ,\cdot }\colon H^{r}(X,\ZlZ)\times H^{3-r}(X,\mu_\ell)\xto{\cup} H^3(X,\Gm)\xrightarrow[\inv]{\sim}\QQ/\ZZ\]
	which induces an isomorphism
	\[H^{3-r}(X,\mu_\ell)\simrightarrow H^{3-r}(X,\ZlZ)^\vee\]
\end{theorem}

Since \(H^i(X,\ZlZ)\) is \(\ell\)-torsion, the image of any homomorphism from the group to \(\QQ/\ZZ\) must lie in \(\fZZZZ{\ell}\). 

We define Bockstein operators \(\delta\) to be the connecting homomorphism coming from the exact sequence of sheaves \(0\rightarrow\mu_\ell\rightarrow\mu_{\ell^2}\rightarrow\mu_\ell\rightarrow0\). 
\[\delta\colon H^i(X,\mu_\ell) \rightarrow H^{i+1}(X,\mu_\ell)\]
Similarly we also define \(\delta'\) to be the Bockstein operator coming from \(0\rightarrow\ZlZ\rightarrow\ZZ/\ell^2\ZZ\rightarrow\ZlZ\rightarrow0\).
\[\delta'\colon H^i(X,\ZlZ) \rightarrow H^{i+1}(X,\ZlZ)\]

These operators are compatible with \(\zeta_*\) such that there is an equality of maps \(\zeta_*\circ\delta'=\delta\circ\zeta_*\). I.e. the following diagram commutes
\[\begin{tikzcd}
H^1(X,\ZZ/\ell\ZZ) \arrow{r}{\delta'}\arrow{d}{\zeta_*}&  H^2(X,\ZZ/\ell\ZZ)\arrow{d}{\zeta_*}
\\H^1(X,\mu_\ell) \arrow{r}{\delta}&  H^2(X,\mu_\ell).
\end{tikzcd}\]

\begin{definition}
	We define the abelian Chern--Simons pairing as follows
	\[\CS(\cdot\,,\cdot)\colon H^1(X,\ZZ/\ell\ZZ)\times H^1(X,\mu_\ell)\rightarrow \frac 1\ell \ZZ/\ZZ\]
	\[(\alpha,\beta) \mapsto \inv(\alpha\cup \delta \beta)\]
\end{definition}

\begin{lemma}[{\cite[Lemma 2.1]{AbelianACS2019}}]
	Given classes \(\alpha\in H^1(X,\ZZ/\ell\ZZ)\) and \(\beta\in H^1(X,\mu_\ell)\), the Bockstein operators satisfies the identity
	\[\delta(\alpha\cup\beta)=\delta'\alpha\cup\beta-\alpha\cup\delta\beta.\]
\end{lemma}

\begin{proof}
	Since \(X\) is a projective variety, it suffices by \cite[10.2]{MilneLecturesEtaleCohomology} to verify the above formula is true in \v{C}ech cohomology.

	Let \(\mathcal U = (U_i)_{i\in I}\) be an \'etale covering of \(X\). We write \(U_{ij}\defeq U_i\times_C U_j\), \(U_{ijk}\defeq U_i\times_C U_j\times_C U_k\), etc. 

	Suppose \(\alpha\) is represented by the \v{C}ech cocycle \((\alpha_{ij})_{i,j\in I}\in Z^1(\mathcal U, \ZlZ)\). In order to compute \(\delta'\alpha\) explicitly, we first pick for every pair \((i,j)\), a lift \(\tilde{\alpha}_{ij}\) of \(\alpha_{ij}\) to \(\ZZ/\ell^2\ZZ\). Then the class of \(\delta'\alpha\) can be represented by the 2-cocycle whose sections are
	\[(\delta'\alpha)_{ijk}\defeq d(\tilde{\alpha})_{ijk} = \tilde{\alpha}_{ij}|_{U_{ijk}}-\tilde{\alpha}_{ik}|_{U_{ijk}}+\tilde{\alpha}_{jk}|_{U_{ijk}}\]
	which takes values in \(\ZlZ\hookrightarrow \ZZ/\ell^2\ZZ\). We can similarly represent \(\delta_1\beta\) as a \v{C}ech cocyle in the same way.

	The cup product \(\alpha\cup\beta\) is represented by the cocycle
	\[(\alpha\cup\beta)_{ijk}=\alpha_{ij}|_{U_{ijk}}\otimes \beta_{jk}|_{U_{ijk}}\]
	which when lifted to \(\ZZ/\ell^2\ZZ \otimes \mu_{\ell^2}\) is the cocyle \(\tilde{\alpha}_{ij}|_{U_{ijk}}\otimes \tilde{\beta}_{jk}|_{U_{ijk}}\). Applying the isomorphism \(\ZZ/\ell^2\ZZ \otimes \mu_{\ell^2}\simeq \mu_{\ell^2}\) given by \(a\otimes b \mapsto a\cdot b\), the cocycle representing \(\delta(\alpha\cup\beta)\) will be
	\[\begin{aligned}
		(\delta(\alpha\cup\beta))_{ijkl}\defeq& 
		(\tilde{\alpha}\cup\tilde{\beta})_{jkl}|_{U_{ijkl}}
		-(\tilde{\alpha}\cup\tilde{\beta})_{ikl}|_{U_{ijkl}}
		+(\tilde{\alpha}\cup\tilde{\beta})_{ijl}|_{U_{ijkl}}
		-(\tilde{\alpha}\cup\tilde{\beta})_{ijk}|_{U_{ijkl}}
		\\=& 
		\left(\tilde{\alpha}_{jk}\cdot\tilde{\beta}_{kl}\right)|_{U_{ijkl}}
		-\left(\tilde{\alpha}_{ik}\cdot\tilde{\beta}_{kl}\right)|_{U_{ijkl}}
		+\left(\tilde{\alpha}_{ij}\cdot\tilde{\beta}_{jl}\right)|_{U_{ijkl}}
		-\left(\tilde{\alpha}_{ij}\cdot\tilde{\beta}_{jk}\right)|_{U_{ijkl}}
		\\=& \left((\tilde{\alpha}_{jk}-\tilde{\alpha}_{ik})\cdot\tilde{\beta}_{kl}\right)|_{U_{ijkl}}+\left(\tilde{\alpha}_{ij}\cdot(\tilde{\beta}_{jl}-\tilde{\beta}_{jk})\right)|_{U_{ijkl}}
		\\=& \left((\tilde{\alpha}_{jk}-\tilde{\alpha}_{ik})\cdot\tilde{\beta}_{kl}\right)|_{U_{ijkl}}+\left(\tilde{\alpha}_{ij}\cdot(\tilde{\beta}_{jl}-\tilde{\beta}_{jk})\right)|_{U_{ijkl}}
		+\left(\tilde{\alpha}_{ij}\cdot\tilde{\beta}_{kl}\right)|_{U_{ijkl}}
		-\left(\tilde{\alpha}_{ij}\cdot\tilde{\beta}_{kl}\right)|_{U_{ijkl}}
		\\=& \left((\tilde{\alpha}_{jk}-\tilde{\alpha}_{ik}+\tilde{\alpha}_{ij})\cdot\tilde{\beta}_{kl}\right)|_{U_{ijkl}}+\left(\tilde{\alpha}_{ij}\cdot(\tilde{\beta}_{jl}-\tilde{\beta}_{jk}-\tilde{\beta}_{kl})\right)|_{U_{ijkl}}
		\\=& \left((\delta'\alpha)_{ijk}\cdot\tilde{\beta}_{kl}\right)|_{U_{ijkl}}+\left(\tilde{\alpha}_{ij}\cdot(\delta_1\beta)_{jkl}\right)|_{U_{ijkl}}.
	\end{aligned}\]
\end{proof}

\begin{lemma}\label{lem:CS_Pairing_Symmetric}
	Upon identifying \(H^{1}(X,\mu_\ell)\) with \(H^{1}(X,\ZlZ)\) via \(\zeta^*\), the Abelian Chern-Simons pairing is symmetric
	\[\CS(\alpha,\beta)=\CS(\beta,\alpha).\]
\end{lemma}

\begin{proof}
	Consider the pro-sheaf \(\ZZ_\ell(1)\defeq \varprojlim_i \mu_{\ell^i}\), there is an exact sequence
	\[0\rightarrow \ZZ_\ell(1)\xrightarrow{\ell} \ZZ_\ell(1) \rightarrow\mu_\ell\rightarrow0\]
	which induces the long exact sequence
	\[\cdots \rightarrow H^2(X,\ZZ_\ell(1))\rightarrow H^2(X,\mu_\ell)\rightarrow H^3(X,\ZZ_\ell(1))\rightarrow\cdots.\]
	
	Since \(H^3(X,\ZZ_\ell(1))= \varprojlim_i H^3(X,\mu_{\ell^i})=\varprojlim_i H^3(X,\mu_{\ell^i})=\ZZ_\ell\) is torsion free, the boundary map \(H^2(X,\mu_\ell)\rightarrow H^3(X,\ZZ_\ell(1))\) is a map from an \(\ell\)-torsion group into a torsion-free group, so must be the zero map. This implies that \(H^2(X,\ZZ_\ell(1))\rightarrow H^2(X,\mu_\ell)\) is surjective.
	
	Since the quotient map \(\ZZ_\ell(1)\rightarrow\mu_\ell\) factors through \(\mu_{\ell^2}\), the map \(H^2(X,\ZZ_\ell(1))\rightarrow H^2(X,\mu_\ell)\) factors through \(H^2(X,\mu_{\ell^2})\), which implies that \(H^2(X,\mu_{\ell^2})\rightarrow H^2(X,\mu_\ell)\) is also surjective. Then from the exact sequence
	\[\cdots \rightarrow H^2(X,\mu_{\ell^2})\rightarrow H^2(X,\mu_\ell) \xrightarrow{\delta_2}H^3(X,\mu_\ell)\rightarrow\cdots\]
	it is deduced that \(\delta_2\colon H^2(X,\mu_\ell)\to H^3(X,\mu_\ell)\) is the zero map. 
	
	Thus, by the previous lemma
	\[\delta'\alpha\cup\beta=\alpha\cup\delta\beta.\]
	Finally we conclude
	\[\CS(\alpha,\beta)=\inv(\alpha\cup\delta\beta)=\inv(\delta'\alpha\cup\beta)=\inv(\beta\cup\delta'\alpha)=\inv(\beta\cup\delta'\alpha)=CS(\beta,\alpha).\]
\end{proof}

\begin{proposition} \label{thm:CS_factors}
	The image of \(\FF_q^\times\) in \(H^1(X, \mu_\ell)\) under the Kummer map lies in the kernel of the Abelian Chern-Simons pairing, and thus the pairing factors into a pairing of the form
	\[\CS\colon \frac{H^1(X,\mu_\ell)}{\FF_q^\times} \times \frac{H^1(X,\mu_\ell)}{\FF_q^\times} \rightarrow \fZZZZ{\ell}.\]
\end{proposition}

\begin{proof}
	By \cref{rmk:Kummer_iso} it suffices to show that \(H^1(\Spec(\FF_q),\ZlZ)\) is in the kernel of the pairing.	Consider the commutative diagram
	\[\begin{tikzcd}
		{H^1(\Spec(\FF_q),\ZlZ)} & {H^1(X,\ZlZ)} \\
		{H^2(\Spec(\FF_q),\ZlZ)} & {H^2(X,\ZlZ)}
		\arrow[hook, from=1-1, to=1-2]
		\arrow["\delta", from=1-1, to=2-1]
		\arrow["\delta"', from=1-2, to=2-2]
		\arrow[from=2-1, to=2-2]
	\end{tikzcd}\]
	where both vertical morphisms are the Bockstein connecting homomorphisms coming from the exact sequence \(0\rightarrow \ZlZ \rightarrow \fZZZZ{\ell} \to \ZlZ \to 0\). But since \(\Spec(\FF_q)\) has cohomological dimension 1, \(H^2(\Spec(\FF_q),\ZlZ)=0\). Thus the Bockstein of the image of \(H^1(\Spec(\FF_q),\ZlZ)\) is equal to zero. 
	
	Thus if \(b\in H^1(\Spec(\FF_q),\ZlZ)\), then
	\[CS(-,b)=0.\] 

	Since we have shown in \cref{lem:CS_Pairing_Symmetric} that the pairing is symmetric, the result follows.
\end{proof}

We now wish to show that under the above identifications, the pairings \(\CS\) and \(A\) agree, but first we will show the following proposition, which is the function field analogue of \cite[Proposition 6.3]{LSTCohenLenstraHeuristicsBilinear}, the statement and proof of the following theorem follows Section 6.1 of \cite{LSTCohenLenstraHeuristicsBilinear} quite closely.

\begin{proposition}\label{Prop: Gm_Pairings_Agree}
	The following two pairings \(H^1(X,\Gm)\times H^1(X,\fZZZZ{\ell})\rightarrow \fZZZZ{\ell}\) are equal:
	\begin{enumerate}
		\item Identify \(H^1(X,\fZZZZ{\ell})\) with \(\Hom(\pi_1^{\ab}(X),\fZZZZ{\ell})\) and then with \(\Hom(\Pic(X),\fZZZZ{\ell})\) via the reciprocity map. 
		Identify \(H^1(X,\Gm)\) with \(\Pic(X)\). Then pair \(\Pic(X)\) and \(\Hom(\Pic(X),\fZZZZ{\ell})\) via the evaluation map. 
		\item Map \(H^1(X,\Gm)\) to \(H^2(X,\mu_\ell)\) via \(\kappa\) the Kummer map, then take the cup product with \(H^1(X,\fZZZZ{\ell})\) to obtain an element of \(H^3(X,\mu_\ell)\), then take the invariant map to obtain an element of \(\fZZZZ{\ell}\). 
	\end{enumerate}
\end{proposition}

\begin{proof}
	Let \(K\) be the function field of \(X\). Fix an element \(\alpha\in H^1(X,\fZZZZ{\ell})\). Since \(H^1(X,\Gm)=\Pic(X)\) is generated by divisors of a single point \(v\), it suffices to check that the pairings agree for every \([v]\) and \(\alpha\). 

	We consider the first pairing. The reciprocity map takes the divisor \([v]\) to \(\Frob_v\in \pi_1^{\ab}(X)\). Then pairing of \(\alpha\) and \([v]\) is given by the action of \(\Frob_v\) on the \(\bar \FF_q\)-points of \(\alpha\). In particular, since \(\Frob_v\) acts on \(K_v\), this pairing uniquely determined by the pairing of \(\Frob_v\) with the pullback \(\alpha_v\in H^1(\Spec \cO_{K_v},\fZZZZ{\ell})\) of \(\alpha\) via the map \(\Spec \cO_{K_v} \to X\). 

	Now we consider the second pairing. Let \(\pi\) be a uniformiser of \(K_v\), the local field at the point \(v\). We first show that by mapping \(\pi\in H^0(K_v,\Gm)\) along the top row of the commutative diagram in \cref{thm:big_comm_diagram}, we get the divisor \([v]\). 

	We let \(U=X\setminus \{v\}\). An element of \(H^1_c(U,\Gm)\) can be expressed as a line bundle on \(U\) with trivialisation at the punctured neighbourhood \(\Spec K_v\) of \(v\). Via this identification, the image of the map \(H^0(v,\Gm)\to H^1_c(U,\Gm)\) sends the element \(\pi\) to the trivial line bundle with identity trivialisation on \(U\), with trivialisation at \(v\) given by multiplication by \(\pi\). 
	
	On the other hand, a line bundle \(L\) on \(U\) with a trivialisation on the punctured neighbourhood \(\Spec K_v\) can be uniquely extended to a line bundle on all of \(X\) via taking only the sections of \(L(U)\) whose image under trivialisation does not have a pole at \(v\). For our line bundle above, the sections of \(\cO_U\) which do not have poles after multiplying by \(\pi\) are precisely those which have at most a simple pole at \(v\), but this is precisely the sheaf corresponding to the divisor \([v]\). 

	Thus by the commutative diagram in \cref{thm:big_comm_diagram}, we have that the Artin-Verdier pairing \(\inv(\kappa[v]\cup \alpha)\) is equal to the invariant map of the local cup product \(\inv_{K_v}(\kappa(\pi)\cup \alpha_v)\). By \cref{lem:local_cup_product} we conclude
	\[\inv(\kappa[v]\cup \alpha) = \inv_{K_v}(\kappa(\pi)\cup \alpha_v) =\alpha_v(\Frob_v).\]
\end{proof}

\begin{lemma}\label{thm:big_comm_diagram}
Let \(\alpha \in H^1(X,\fZZZZ{\ell})\), \(v\) be a place of \(X\) with punctured neighbourhood \(\Spec K_v\), \(\alpha_v\in H^1(\cO_{K_v},\fZZZZ{\ell})\) to be the pullback of \(\alpha\) via \(\Spec \cO_{K_v}\to X\), and \(U=X-\{v\}\). Then the following diagram commutes:

\[\begin{tikzcd}
	{H^0(K_v,\Gm)} & {H^1_c(U,\Gm)} & {H^1(X,\Gm)} \\
	{H^1(K_v,\mu_\ell)} & {H^2_c(U,\mu_\ell)} & {H^2(X,\mu_\ell)} \\
	{H^2(K_v,\mu_\ell)} & {H^3_c(U,\mu_\ell)} & {H^3(X,\mu_\ell)} \\
	\fZZZZ{\ell} & \fZZZZ{\ell} & \fZZZZ{\ell}
	\arrow[from=1-1, to=1-2]
	\arrow["\kappa", from=1-1, to=2-1]
	\arrow[from=1-2, to=1-3]
	\arrow["\kappa", from=1-2, to=2-2]
	\arrow["\kappa", from=1-3, to=2-3]
	\arrow[from=2-1, to=2-2]
	\arrow["{\cup \alpha_v}", from=2-1, to=3-1]
	\arrow[from=2-2, to=2-3]
	\arrow["{\cup \alpha}", from=2-2, to=3-2]
	\arrow["{\cup \alpha}", from=2-3, to=3-3]
	\arrow[from=3-1, to=3-2]
	\arrow["\inv", equals, from=3-1, to=4-1]
	\arrow[from=3-2, to=3-3]
	\arrow["\inv", equals, from=3-2, to=4-2]
	\arrow["\inv", equals, from=3-3, to=4-3]
	\arrow[equals, from=4-1, to=4-2]
	\arrow[equals, from=4-2, to=4-3]
\end{tikzcd}\]

Where \(\kappa\) is the Kummer map arising from the Kummer exact sequence, the horizontal maps in the left column arise from the exact sequence of compactly supported cohomology in \cite[II.2.3(a)]{MilneADT}, and the horizontal maps in the right column arise from \cite[II.2.3(d)]{MilneADT} and the fact that \(H^r_c(X,\cF)=H^r(X,\cF)\) since \(X\) is compact.
\end{lemma}

\begin{proof}
	The compactly supported cohomology groups \(H^\star_c(U,\cF)\) are defined as a shifted mapping cone of the localisation morphism of \v{C}ech cochains
	\(\loc\colon C^\bullet(U,\cF)\to C^\bullet(K_v,\cF_v)\), or in other words the compactly supported cohomology groups \(H^r_c(U,\cF)\) are the cohomology groups of the complex
	\[\cone^\bullet(\loc[-1])= C^\bullet(U,\cF)\oplus C^\bullet(K_v,\cF_v)[-1].\]
	So on the level of cochains, the left column horizontal maps \(H^r(K_v,\cF_v)\rightarrow H^{r+1}_c(U,\cF)\) is given by the inclusion map into the local component of the mapping cone. 

	On the other hand the right column horizontal arrows are obtained on the level of cochains via the composition
	\[C^\bullet(U,\cF)\oplus C^\bullet(K_v,\cF_v)[-1]\to C^\bullet(U,\cF)\to C^\bullet(X,\cF).\]
	On the level of cochains, the connecting homomorphism \(\kappa\) is obtained via the composition of inverse image and differential maps; and the cup product with \(\alpha\) can be interpreted as a tensor product of \v{C}ech cocycles. All of these maps  commutes with the inclusions, projections and pullbacks, and so the diagram commutes. 
\end{proof}

\begin{lemma}\label{lem:local_cup_product}
	Let \(K=K_v\) be a non-archimedean local field. Let \(G=\pi_1(\Spec K_v)= \Gal(\bar{K_v}/K_v)\) and \(\Gamma = \pi_1(\Spec \cO_{K_v})= \Gal(K_v^{ur}/K_v)\) be the absolute Galois group and maximal unramified Galois group of \(K_v\) respectively. Suppose we are given a uniformiser \(\pi\in H^0(K_v,\Gm)\) and an element \(\alpha\in H^1(\cO_{K_v},\fZZZZ{\ell})\), then
	\[\inv(\kappa(\pi)\cup\alpha)=\alpha(\Frob_v).\]
	Here \(\alpha \in H^1(\cO_{K_v},\fZZZZ{\ell}) =\Hom(\Gamma,\fZZZZ{\ell})\) is viewed as a homomorphism from \(\Gamma\) to \(\fZZZZ{\ell}\). 
\end{lemma}

\begin{proof}
	Let \(\pi_0\) be an \(\ell\)th root of \(\pi\) so that \(\pi_0^\ell=\pi\). Then a cocycle representing \(\kappa(\pi)\in H^1(K_v,\mu_\ell)\) is \(\sigma \mapsto \frac{\sigma(\pi_0)}{\pi_0}\).
	
	Let \(\phi\) be the cocycle representing \(\kappa(\pi)\cup\alpha \in H^2(K_v,\mu_\ell)\). Since the map \(\fZZZZ{\ell}\times \mu_\ell\to \mu_\ell\) is given by \((q,\zeta) \mapsto \zeta^{\ell q}\), it follows that the cocycle \(\phi\) is given by
	\[\phi(\sigma,\tau)=\left(\frac{\sigma(\pi_0)}{\pi_0}\right)^{\ell\alpha(\tau)}.\]

	Given an element \(\sigma\in \Gamma\), let \(n_\sigma\) denote the unique element of \(\{0,1,\dots, \ell-1\}\) that is congruent to \(\ell\alpha(\sigma)\). This definition can also be naturally be extended to \(\sigma\in G\) via quotienting by the inertia subgroup first. 

	On the other hand, define the cocycle \(\psi\colon \Gamma^2\rightarrow K_v^{ur \times}\) giving a class \(H^2(\Gamma, K_v^{ur \times})\) via
	\[\psi(\sigma,\tau)=\begin{cases}
		1 & \text{if } n_\sigma + n_\tau < \ell
		\\ \pi & \text{if } n_\sigma + n_\tau \geq \ell
	\end{cases}.\]
	
	We define a cocycle \(c\colon G\to \bar{K}_v^\times\) via \(c(\sigma)=\pi_0^{n_\sigma}\). Then the coboundary of \(c\) is
	\[\begin{aligned}
		dc(\sigma,\tau) &= \frac{\sigma c(\tau)c(\sigma)}{c(\sigma\tau)}
		\\&= \frac{\sigma(\pi_0^{n_\tau})\pi_0^{n_\sigma}}{\pi_0^{n_{\sigma\tau}}}
		\\&= 
		\begin{cases}
			\dfrac{\sigma(\pi_0)^{n_\tau}}{\pi_0^{n_\tau}} & \text{if } n_\sigma + n_\tau < \ell
			\\[10pt]\pi\dfrac{\sigma(\pi_0)^{n_\tau}}{\pi_0^{n_\tau}} & \text{if } n_\sigma + n_\tau \geq \ell
		\end{cases}
		\\&= \phi \inf \psi.
	\end{aligned}\]

	Thus \([\phi]=(\inf[\psi])^{-1}\). Thus it suffices to compute \(\inv(\inf([\psi]))\). 

	Now, from \cite[p. 130]{CFAlgebraicNumberTheory} the invariant map \(\inv\colon H^2(G,\bar{K}_v^\times)\rightarrow \QZ\) is defined as the composition
	\[H^2(G,\bar{K}_v^\times)\xleftarrow[\sim]{\inf}H^2(\Gamma, K_v^{ur \times})\xto{\nu}H^2(\Gamma,\ZZ)\xot[\sim]{\rho}H^1(\Gamma,\QZ)\xto[\sim]{\gamma}\QZ.\]
	Where:
	\begin{itemize}
		\item \(\nu\) is the map on cohomology induced by the valuation map \(K_v^{ur \times}\to \ZZ\).
		\item \(\rho\) is the connecting homomorphism of the short exact sequence \(0\to\ZZ\to\QQ\to\QZ\to0\).
		\item \(\gamma\) is the homomorphism \(\Hom(\Gamma,\QZ)\to\QZ\) given by evaluating the homomorphism at the Frobenius element \(\Frob_v\in \Gamma\). 
	\end{itemize}

	Thus we wish to compute \(\gamma\circ\rho^{-1}\circ\nu([\psi])\), firstly we see that
	\[\nu(\psi)(\sigma,\tau)=\begin{cases}
		0 & \text{if } n_\sigma + n_\tau < \ell
		\\ 1 & \text{if } n_\sigma + n_\tau \geq \ell.
	\end{cases}\]

	On the other hand, by definition \(\alpha \in \Hom(\Gamma,\fZZZZ{\ell})\) satisfies \(\gamma(\alpha) = \alpha(\Frob_v)\), so it just suffices to prove that \(\rho(\alpha)=\nu(\psi)\). Note that the cocycle \(a\colon\Gamma\to\QQ\) given by \(a(\sigma)=\frac{n_\sigma}{\ell}\) is a lift of \(\alpha\), and \(a\) has coboundary
	\[\begin{aligned}
		da(\sigma,\tau)&= \sigma a(\tau)+a(\sigma)-a(\sigma\tau)
		\\&=\begin{cases}
			0 & \text{if } n_\sigma + n_\tau < \ell
			\\ 1 & \text{if } n_\sigma + n_\tau \geq \ell
		\end{cases}
		\\&= \nu(\psi).
	\end{aligned}\]
\end{proof}

Finally, the main result of this subsection is the following theorem, showing that the Abelian Chern-Simons pairing agrees with the Class Field Theory pairing.

\begin{theorem}\label{thm:Pairings_Agree}
	Under the isomorphism \(\bar{s}\colon \frac{H^1(X,\mu_\ell)}{\FF_q^\times}\simrightarrow J[\ell](\FF_q)\) defined in \cref{thm:CS_factors}, the pairings \(A\) and \(\CS\) agree. In other words the following diagram commutes

	\[\begin{tikzcd}
		{\frac{H^1(X,\mu_\ell)}{\FF_q^\times}} & {\frac{H^1(X,\mu_\ell)}{\FF_q^\times}} && {\fZZZZ{\ell}} \\
		{J[\ell](\FF_q)} & {J[\ell](\FF_q)} && {\fZZZZ{\ell}.}
		\arrow["\times"{description}, draw=none, from=1-1, to=1-2]
		\arrow["{\bar{s}}", from=1-1, to=2-1]
		\arrow["\CS"', from=1-2, to=1-4]
		\arrow["{\bar{s}}", from=1-2, to=2-2]
		\arrow[from=1-4, to=2-4]
		\arrow["\times"{description}, draw=none, from=2-1, to=2-2]
		\arrow["A"', from=2-2, to=2-4]
	\end{tikzcd}\]
\end{theorem}

\begin{proof}
	Let \(\iota\) denote the map \(H^r(X,\mu_\ell)\to H^r(X,\Gm)\) induced by the inclusion \(\mu_\ell\hookrightarrow \Gm\).

	Given elements \(\alpha,\beta\in H^1(X,\mu_\ell)\), we see that the pairing \(A(s(\alpha),s(\beta))\) can be identified with pairing (1) of \(\iota(\alpha)\) and \(\iota(\beta)\) in \cref{Prop: Gm_Pairings_Agree}. 
 
	On the other hand, observe that by functoriality, the map of exact sequences
	\[\begin{tikzcd}[cramped]
		0 & {\mu_\ell} & {\mu_{\ell^2}} & {\mu_\ell} & 0 \\
		0 & {\mu_\ell} & \Gm & \Gm & 0
		\arrow[from=1-1, to=1-2]
		\arrow[from=1-2, to=1-3]
		\arrow[equals, from=1-2, to=2-2]
		\arrow[from=1-3, to=1-4]
		\arrow[hook, from=1-3, to=2-3]
		\arrow[from=1-4, to=1-5]
		\arrow[hook, from=1-4, to=2-4]
		\arrow[from=2-1, to=2-2]
		\arrow[from=2-2, to=2-3]
		\arrow[from=2-3, to=2-4]
		\arrow[from=2-4, to=2-5]
	\end{tikzcd}\]
	induces the following commutative diagram of cohomology:
	\[\begin{tikzcd}[cramped]
		{H^r(X,\mu_\ell)} && {H^r(X,\mu_\ell)} \\
		& {H^r(X,\Gm)}
		\arrow["\delta", from=1-1, to=1-3]
		\arrow["\iota",from=1-1, to=2-2]
		\arrow["\kappa", from=2-2, to=1-3]
	\end{tikzcd}\]
	Where the Bockstein map factors through the Kummer map. Then it is easy to see that the pairing \(\CS(\alpha,\beta)\) is equal to the pairing (2) of \(\iota(\alpha)\) and \(\iota(\beta)\) in \cref{Prop: Gm_Pairings_Agree}.

	Finally, by \cref{Prop: Gm_Pairings_Agree}, these two pairings agree. 
\end{proof}

An immediate corollary is the following:

\begin{corollary}\label{cor:A_Symmetric}
    The pairing \(A\) is symmetric.
\end{corollary}

\begin{proof}
	It follows from \cref{thm:Pairings_Agree} and \cref{lem:CS_Pairing_Symmetric} that \(A\) is symmetric.
\end{proof}

\subsection{Computation of the Arithmetic Path Integral}

\begin{theorem}\label{thm:Path_Integral}
    Suppose \(J(\FF_q)\) has no points of order \(\ell^2\) and that \(\mu_{\ell^2}\subset \FF_q\). Then the arithmetic path integral evaluates to
	\[\sum_{\gamma \in J[\ell](\FF_q)} e^{2\pi i A(\gamma)} = \sqrt{|J[\ell](\FF_q)|} \legendre{\det A}{\FF_\ell} \left(i^{(\ell-1)^2/4}\right)^{\dim_\ell J[\ell](\FF_q)}.\]

	If additionally the Frobenius \(\Fr_q\) acts semisimply on \(J[\ell]\), then we can express the path integral as
	\[\sum_{\gamma \in J[\ell](\FF_q)} e^{2\pi i A(\gamma)} = \sqrt{|J[\ell](\FF_q)|} \legendre{(-1)^{(\dim_\ell J[\ell](\FF_q))/2}\det(A)}{\FF_\ell}.\]
\end{theorem}

\begin{proof}
	By \cref{prop:A_nondegen_bilinear} and \cref{cor:A_Symmetric} \(A\) is a non-degenerate symmetric bilinear form, so we can write the sum as a Gaussian integral over a finite field of the form
	\[\sum_{x\in\FF_\ell^n}\exp\left[\frac{2\pi i}{l}xQx^T\right].\]
	Here, \(x\) is a vector representing \(\gamma\), \(Q\) is the matrix representing \(A\) and \(n=\dim_\ell J[\ell](\FF_q)\). The factor of \(\frac 1\ell\) comes from the fact that \(A\) has image in \(\frac 1l\ZZ/\ZZ\) rather than \(\ZlZ\). 

	Then applying the theorem in \cite[Chapter 9, Theorem 3.1]{Neretin2011LecturesOG}, we obtain
	\[\begin{aligned}
		\sum_{\gamma \in J[\ell](\FF_q)} e^{2\pi i A(\gamma)}&= \ell^{\frac n2} \legendre{\det{A}}{\FF_\ell} \left(i^{(\ell-1)^2/4}\right)^{n}
		\\&= \sqrt{|J[\ell](\FF_q)|} \legendre{\det A}{\FF_\ell} \left(i^{(\ell-1)^2/4}\right)^{\dim_\ell J[\ell](\FF_q)}
	\end{aligned}\]

	Noting that
	\[i^{(\ell-1)^2/4}=\begin{cases}
	1, & \ell\equiv 1 \mod 4
	\\ i, & \ell\equiv 3\mod 4
	\end{cases}\]
	and thus in particular when \(\ell\equiv 1\mod 4\), this path integral is real and simply evaluates to
	\[\sqrt{|J[\ell](\FF_q)|} \legendre{\det A}{\FF_\ell}.\]

	On the other hand, when \(\ell\equiv 3\pmod 4\), we note that if the Frobenius \(\Fr_q\) acts semisimply on \(J[\ell]\) then by \cref{prop:eval_multiplicity} that \(\dim_\ell J[\ell](\FF_q)\) is actually even, so the path integral is also real, and can be written as
	\[(-1)^{(\dim_\ell J[\ell](\FF_q))/2}\sqrt{|J[\ell](\FF_q)|} \legendre{\det A}{\FF_\ell}.\]

	Since \(\displaystyle\legen{-1}{\FF_\ell}=\begin{cases}
	1 &\ell\equiv 1\mod 4
	\\ -1&\ell\equiv 3\mod 4
	\end{cases}\), we can combine these two expressions to obtain
	\[\sum_{\gamma \in J[\ell](\FF_q)} e^{2\pi i A(\gamma)} = \sqrt{|J[\ell](\FF_q)|} \legendre{(-1)^{(\dim_\ell J[\ell](\FF_q))/2}\det(A)}{\FF_\ell}.\]
\end{proof}

\section{Proving the Main Theorem} \label{sec:Conclusion}

Finally, we combine the results of \cref{thm:Path_Integral} and \cref{thm:TraceFormula} to obtain our main theorem with all the signs determined.

\begin{theorem}\label{thm:MainTheoremFull}
Let \(J\) be the Jacobian of a genus \(g\) curve \(X\) over a finite field \(\FF_q\). 
	Let \(\ell\) be an odd prime satisfying \(q\equiv 1\pmod {\ell^2}\). If \(\Fr_q\) acts semisimply on the \(\FF_\ell\)-vector space \(J[\ell]\), and \(J(\FF_q)\) has no points of order \(\ell^2\), then we have the following equality
	\[\tr(\Fr_q|\cH) =\legendre{(-1)^g \det'(1-\Fr_q) \det(A)}{\FF_\ell}\sum_{\gamma\in J[\ell](\FF_q)}e^{2\pi iA(\gamma)},\]

	where \(\det'(1-\Fr_q)\) is the regularised determinant of the linear map \(1-\Fr_q\) on vector space \(J[\ell]\). 
\end{theorem}

\begin{proof}
	By combining \cref{thm:Path_Integral,thm:TraceFormula} we have the following equality where both sides are equal to \(\sqrt{|J[\ell](\FF_q)|}\),
	\[\legen{(-1)^{g-\dim_{\ell} J[\ell](\FF_q)/2}\det'(1-\Fr_q)}{\FF_\ell}\tr(\Fr_q|\cH) =\legendre{(-1)^{(\dim_\ell J[\ell](\FF_q))/2}\det(A)}{\FF_\ell}\sum_{\gamma\in J[\ell](\FF_q)}e^{2\pi iA(\gamma)}.\]

	Rearranging the Legendre symbols, we obtain the desired result.
\end{proof}

\section{Computational Examples and Remarks}\label{sec:Examples}

In this section we provide two explicit examples of \cref{thm:MainTheoremFull} for an elliptic curve \(X/\FF_q\), in the case where \(q=19, \ell=3\). 

Let us fix an isomorphism \(\zeta\colon\{1,7,11\}=\mu_3(\FF_{19})\simeq \FF_3\) by sending \(7\in \mu_3(\FF_{19})\) to \(1\in \FF_3\). We also fix the isomorphism \(\psi\colon \FF_3\simeq \mu_3(\CC)\) via \(1\mapsto \xi = e^{\frac{2\pi i}{3}}.\)

\begin{example}
	Let \(X=E\) be the curve \[E: y^2 = x^3 - 5x /\FF_{19}.\]
	
	We can check that the only \(\FF_{19}\)-rational 3-torsion point is the origin \(O_E\), and that the 3-torsion group \(E[3]\) is defined over the extension \(\FF_{19^4}/\FF_{19}\). Explicitly, let \(\FF_{19^4}=\FF_{19}(\alpha)\), where \(\alpha\in \FF_{19^4}\) satisfies the minimal polynomial \[\alpha^4+\alpha^3+2\alpha^2+15\alpha+3=0.\]

	Then using the \verb:torsion_basis(): function on SageMath, the 3-torsion \(E[3]\) is generated by the elements 
	\[\begin{aligned}
		P_1&=(18\alpha^3 + \alpha + 17, 7\alpha^3 + 15\alpha^2 + \alpha + 15)
	\\	P_2&=(18\alpha^3 + \alpha + 11, 14\alpha^3 + 4\alpha^2 + 9\alpha + 17).
	\end{aligned}\]

	By a straightforward calculation we can see that the Frobenius acts on the basis elements in the following manner
	\[\begin{aligned}
		\Fr_{19}P_1&=P_1+2P_2
		\\\Fr_{19}P_2&=2P_1+2P_2.
	\end{aligned}\]

	Moreover, let \(\langle-,-\rangle\) denote the Weil pairing on \(E[3]\), via the \verb:weil_pairing(): function on SageMath, we compute that \[\langle P_1,P_2\rangle = 7\in \mu_3(\FF_{19}),\]
	which we will identify with \(\langle P_1,P_2\rangle= 1\in \FF_3\). From now on the values of the Weil pairing \(\langle -,-\rangle\) refer to the values in \(\FF_3\) after identifying \(\zeta\colon\mu_3(\FF_{19})\simeq \FF_3\). 

	Since \(E[3]\) is a 2-dimensional symplectic space with respect to the Weil pairing, every 1-dimensional subspace is Lagrangian. However, the matrix \(\left(\begin{smallmatrix}1&2\\2&2\end{smallmatrix}\right)\) is not diagonalisable over \(\FF_3\), so there does not exist any Frobenius-invariant Lagrangian. Let us pick the Lagrangian \(M=\FF_3P_1\), the 1-dimensional subspace generated by \(P_1\). We also pick the complement \(M'=\FF_3P_2\), so that \(E[3]=M\oplus M'\). 

	We now evaluate \(\tr(\Fr_{19}| C_{M^\circ})\) via \cref{thm:Trace_Any_Lagrangian}. 
	
	First we compute the constant \(A_{M^\circ,\Fr_{19} M^\circ}\). Fix the orientation \(o_M=P_1\). 
	Note that \(M\) and \(\Fr_{19} M\) are transverse, so \(I=M\cap \Fr_{19} M=\{O_E\}\) is 0-dimensional, and \(n_I=1\). 
	We also have \(o_{\Fr_{19} M}=\Fr_{19} P_1=P_1+2P_2\), and since \(I\) is trivial we have that \(o_{M/I}=o_M\) and \(o_{\Fr_{19} M/I}=o_{\Fr_{19} M}\). We also compute that 
	\[G(\tfrac{1}{2},3)= \sum_{x\in \FF_3}\psi(\tfrac{1}{2}x^2)=1+\xi^2+\xi^2=-i\sqrt{3}.\]

	Thus we have
	\[\begin{aligned}
		A_{M^\circ,\Fr_{19} M^\circ}&=(G(\tfrac{1}{2},3)/3)^{1} \left(\frac{(-1)^{\binom{1}{2}} \langle P_1+2P_2, P_1\rangle}{\FF_3}\right).
		\\&=\frac{-i\sqrt 3}{3} \left(\frac{1}{\FF_3}\right)
		\\&=\frac{-i\sqrt 3}{3}.
	\end{aligned}\]
	
	Now we compute the sum portion of the formula in \cref{thm:Trace_Any_Lagrangian}.
	
	We note that \(\cal S = M'\), since \(M+\Fr_{19}M = E[3]\) spans the entire space. For each \(x\in M' = \{O_E,P_2, 2P_2\}\), we wish to find elements \(m_{x}, n_{x}\in M\) such that \[\Fr_{19}x-x=m_{x}+\Fr_{19}n_{x},\]
	and we wish to evaluate the sum \(\sum_{x\in \mathcal S} \psi\left(\frac12\langle m_x+n_x,x\rangle\right)\). 

	By inspection, we can make the following choices of \(m_x,n_x\) and obtain the following values of \(\langle m_x+n_x,x\rangle\)
	\[\left\{\begin{aligned}
		m_{P_2}=O_E,&\quad n_{P_2}=2P_1,& \langle m_{P_2}+n_{P_2}, P_2\rangle= \langle 2P_1,P_2\rangle=2
		\\m_{2P_2}=O_E,&\quad n_{2P_2}=P_1,& \langle m_{2P_2}+n_{2P_2}, 2P_2\rangle= \langle P_1,2P_2\rangle=2
		\\m_{O_E}=O_E,&\quad n_{O_E}=O_E,&\langle m_{O_E}+n_{O_E}, O_E\rangle= \langle O_E,O_E\rangle=0
	\end{aligned}\right..\]

	Thus we evaluate
	\[\sum_{x\in \mathcal S} \psi\left(\frac12\langle m_x+n_x,x\rangle\right)=\xi+\xi+1=i\sqrt3.\]

	Finally, combining the two computations, we have that the trace is equal to 
	\[\begin{aligned}
	\tr(\Fr_{19}| C_{M^\circ})&=A_{M^\circ,\Fr_{19} M^\circ}\sum_{x\in \mathcal S} \psi\left(\frac12\langle m_x+n_x,x\rangle\right)
	\\&=\frac{-i\sqrt 3}{3}(i\sqrt3)
	\\&=1.
	\end{aligned}\]

	On the other hand, since \(E[3](\FF_{19})=\{O_E\}\) is zero-dimensional, the map \(A\) must just be the zero map, and thus the path integral trivially evaluates to 1. 

	We finally check that our computation indeed satisfies \cref{thm:MainTheoremFull}, we check that the regularised determinant \(\det'(1-\Fr_{19})\) is just equal to the determinant \(\det(1-\Fr_{19})=\det(\begin{smallmatrix}0&1\\1&2\end{smallmatrix})=-1\). 
	
	Indeed, the trace and path integral in this example are both equal to 1, and thus differ by the sign 
	\[\legendre{(-1)^g\det'(1-\Fr_{19})}{\FF_3}=\legendre{1}{\FF_3}=1.\]
\end{example}

\begin{remark}
	In the above example, the trace computation was quite involved while the path integral was trivial since \(E[3]\) had no non-trivial \(\FF_{19}\)-rational points. We will now provide a second example where the trace computation is simple due to the existence of an invariant Lagrangian, but the path integral computation is more intricate. 

	We also remark that in the elliptic curve case it is impossible for both the trace and the path integral to be `interesting'. Due to our assumptions of the semisimplicity of Frobenius and \(\mu_3\subset \FF_{19}\), it follows \(E[3](\FF_{19})\) must be an even dimensional linear subspace of \(E[3]\), which means it must be trivial or the entire space. If it is zero dimensional, then the path integral is trivial; if it is the entire space, then \(\Fr_{19}\) must act trivially on \(E[3]\), and every Lagrangian is Frobenius-invariant. 

	For higher genus curves there will exist examples where both the trace and path integral are non-trivial computations. 
\end{remark}

\begin{example}
	Let \(X=E\) be the curve \[E: y^2=x^3+12x-1/\FF_{19}.\]

	In this case, we notice that the 3-torsion \(E[3]\subset E(\FF_{19})\) is \(\FF_{19}\)-rational. This means that the Frobenius \(\Fr_{19}\) acts trivially on the \(\FF_3\) vector space \(E[3]\), and any Lagrangian \(M\) would be fixed by the Frobenius action. 

	By \cref{Cor: Invariant Lagrangian}, it follows that the trace is equal to \[\tr(\Fr_{19}|C_M)=\legendre{\det(\Fr_{19}|M)}{\FF_3}\sqrt{|E[3](\FF_{19})|}=3.\]

	Now we move onto the path integral calculation. We use the \verb:torsion_basis(): function on SageMath again, to find the following basis for \(E[3]\) 
	\[\begin{aligned}
		P_1&=(10,13)
		\\P_2&=(15,18).
	\end{aligned}\]

	Let \(\alpha_i\) denote the \(H^1(E,\ZZ/3\ZZ)=\Hom(\pi_1(E),\ZZ/3\ZZ)\) class associated to \(P_i\). We wish to compute \(\alpha_i\) explicitly. 
	
	% The Abel-Jacobi map associates the point \(P_i\) with the divisor \([P_i]-[O_E]\).
	
	We use the function \verb:basis_function_space(): on SageMath to compute that
	\[\begin{aligned}
	\div (y + 7x + 12)=3[P_1]-3[O_E]
	\\\div (y + 11x + 7)=3[P_2]-3[O_E]
	\end{aligned}\]

	Let \(f_1=y + 7x + 12\) and \(f_2=y + 11x + 7\), then the cohomology class \(\alpha_i\) corresponds to the extension of global function fields \(L_i\defeq \FF_{19}(E)[\sqrt[3]{f_i}]/\FF_{19}(E)\). In particular, \(\alpha_i\) corresponds to the homomorphism
	\[\alpha_i\colon \sigma \mapsto \zeta\left( \frac{\sigma(\sqrt[3]{f_i})}{\sqrt[3]{f_i}}\right).\]

	Explicitly, the field \(L_i\) has presentation 
	\[L_i=\rm{Frac} \,\FF_{19}[x,y,z]/(y^2-x^3-12x+1,z^3-f_i).\]

	On the other hand, let \(\Fr_{P_i} = \rm{Res}\, P_i\in \pi_1^{\ab}(E)\) denote the images of \(P_i\) under the reciprocity map. We wish to explicitly write down the action of \(\alpha_i\) when restricted to the field extensions \(L_j\), for \(i,j\in \{1,2\}\). We will now compute one of the cases in detail below. 

	\textbf{Case \(i=1,j=2\):} Let \(Q\) be a place of \(L_1\) lying over \(P_2\). We note that since \(L_i\) is an unramified extension of \(\FF_{19}(E)\), it follows that \(t=x-15\) is a local parameter for \(Q\), so the local field \((L_1)_{Q}\) can be written as \(\FF_{19^3}((t))\), where \(\FF_{19^3}\) is the residue field. Then we can write \(y\) as the power series
	\[y=\sqrt{(t+15)^3+12(t+15)-1}=18 + 8t + 9t^3 + 15t^4+\cdots.\]
	Here, we picked \(y\) to be the square root with constant term \(18\) since it vanishes at the point \(P_2=(15,18)\).

	Then, \(z\) is represented by the power series 
	\[\begin{aligned}
	z&=\sqrt[3]{y + 7(t+15) + 12}
	\\&= \sqrt[3]{2}+ 12\sqrt[3]{2}t+8\sqrt[3]{2}t^2 +\cdots
	\end{aligned}\]

	where \(\sqrt[3]{2}\in \FF_{19^3}\) generates the residue field extension \(\FF_{19^3}/\FF_{19}\). Applying the Frobenius element to the coefficients of \(z\) we obtain
	\[\begin{aligned}
		\Fr_{P_2}(z)&=7\sqrt[3]{2} + 8\sqrt[3]{2}t +18\sqrt[3]{2}t^2 + \cdots
		\\&= 7z.
	\end{aligned}\]

	Thus we conclude that 
	\[A(P_1,P_2)=\alpha_1(\Fr_{P_2})=\zeta\left(\frac{\Fr_{P_2}(z)}{z}\right)=\zeta(7)=1\in \FF_3.\]

	\textbf{Other Cases:} Via the same method, we also compute that \(A(P_1,P_1)=0\), \(A(P_2,P_2)=2\) and \(A(P_2,P_1)=1\).

	Thus the pairing \(A\) acts via the matrix \(\begin{pmatrix}0&1\\1&2\end{pmatrix}\), and the quadratic form associated to it is \[A(aP_1+bP_2)=2ab+2b^2=\tfrac12(a+2b)^2-\tfrac12 a^2= a^2-(b-a)^2.\]
	
	Now we can evaluate the path integral

	\[\begin{aligned}
	\sum_{\gamma\in E[\ell](\FF_{19})}e^{\frac{2\pi iA(\gamma)}{3}}
	=\sum_{a,b\in \FF_3}\xi^{a^2-(b-a)^2}
	=\sum_{a\in \FF_3}\sum_{c\in \FF_3}\xi^{a^2}\xi^{-c^2}
	&=\left(\sum_{a\in \FF_3}\xi^{a^2}\right)\left(\sum_{c\in \FF_3}\xi^{-c^2}\right)
	\\&=(1+\xi+\xi)(1+\xi^2+\xi^2)
	\\&=(i\sqrt{3})(-i\sqrt{3})
	\\&= 3.
	\end{aligned}\]

	Finally, we check that the trace and the path integral computed satisfies \cref{thm:MainTheoremFull}. 

	We see that \(\det A = -1\), and since \(1-\Fr_{19}\) is identically the zero map on \(E[3]\), so the regularised determinant is equal to  \(1\). Since both the trace and the path integral evaluate to \(3\) in this example, they indeed differ by the sign
	\[\legendre{(-1)^g \det'(1-\Fr_q) \det(A)}{\FF_\ell}=\legendre{(-1)(1)(-1)}{\FF_3}=1.\]
\end{example}

% \subsection*{Further Related Questions}\label{sec:FurtherResearch}

\begin{remark}
We highlight possible further research directions related to trace--path integral formulae in arithmetic.

\begin{enumerate}
	\item \textbf{An intrinsic proof of the trace--path integral formula.} The current proof evaluates each side of the formula separately and compares them. A more direct proof that shows the trace and the path integral are intrinsically related could be more enlightening.
	\item \textbf{Trace--path integral formulae for number fields.} By Weil's trichotomy of analogies between number fields, function fields, and manifolds, we should expect a similar formula to exist in the case of number fields. While number fields do not have a Frobenius action, but given a cyclotomic \(\ZZ_p\)-extension \(K_\infty/K\), the \(\gamma\) action for a topological generator \(\gamma\in \Gal(K_\infty/K)\) can be viewed as the analogous action.
		\begin{center}
		\begin{tabular}{ccc}
		Number Fields& Function Fields & Manifolds\\\hline
		Number Ring \(\Spec \OK\)& Smooth Projective Curve \(X/\FF_q\)& 3-manifold \(M\)\\
		\(\ZZ_p\)--extension \(\Spec \cO_{K[\mu_{p^\infty}]}\)&Curve over Algebraic Closure \(\bar{X}/\bar{\FF}_q\)&Riemann Surface \(\Sigma\)\\
		\(\gamma\) action&Frobenius Action &Monodromy Action\\
		\end{tabular}
		\end{center}
	
		A very interesting further research question would be to prove a trace--path integral formula over number fields, where a path integral over \(X=\Spec \OK\) is expressed as the trace an action \(\gamma\) on a space \(\cH\), analogous to the quantisation of a space of fields on  \(\Spec \cO_{K[\mu_{p^\infty}]}\).
	\item \textbf{Trace--path integral formulae for other actions.} The trace--path integral formula proved in this paper is for the Abelian Arithmetic Chern--Simons action. In theory one should expect similar formulae to hold for path integrals of other arithmetic actions, such as non-abelian Arithmetic Chern--Simons actions, or the arithmetic BF action defined in \cite{CKNoteAbelianArithmetic}. 

	\item \textbf{Sheaf-theoretic versions of trace--path integral formulae.} By studying how the Trace and Path Integrals deform upon varying the curve, one should expect a Trace-Path integral formula which relates to the universal theta divisor on the moduli space $\mathcal{M}_g$ of genus $g$ curves, where the work of this paper can be viewed as the case on a single point in $\mathcal{M}_g$. 
\end{enumerate}
\end{remark}

\printbibliography
 
\end{document}